\documentclass[12pt, reqno, a4paper]{amsart}
\usepackage{graphicx, epsfig, psfrag}
\usepackage{amsfonts,amsmath,amssymb,amsbsy,amsthm}
\usepackage{bm}
\usepackage{color}
\usepackage{float}
\usepackage{mathrsfs}

\usepackage{pdfsync}

\usepackage{environments}

\columnwidth\textwidth

\definecolor{cocol}{rgb}{0.7, 0, 0}
\definecolor{mdcol}{rgb}{0, 0.5, 0}
\definecolor{ascol}{rgb}{0, 0, 0.7}

\renewcommand{\cases}[1]{\left\{ \begin{array}{rl} #1 \end{array} \right.}
\newcommand{\smfrac}[2]{{\textstyle \frac{#1}{#2}}}

\def\R{\mathbb{R}}
\def\N{\mathbb{N}}
\def\Z{\mathbb{Z}}
\def\C{\mathbb{C}}

\def\CC{{\rm C}}

\def\pp{\partial}

\def\<{\langle}
\def\>{\rangle}

\def\eps{\varepsilon}

\def\qcf{{\rm qcf}}
\def\qnl{{\rm qnl}}
\def\a{{\rm a}}
\def\c{{\rm c}}
\def\symm{{\rm sym}}

\def\E{\mathcal{E}}

\def\Lqcf{L^\qcf}
\def\Lqcfp{L^\qcf_0}
\def\Lqnl{L^\qnl}
\def\Vqcf{V^\qcf}
\def\Vqnl{V^\qnl}
\def\Vsymm{V^\symm}
\def\Wsymm{\tilde{V}^\symm}
\def\Lamtil{\tilde{\Lambda}}
\def\La{L^\a}
\def\Lc{L^\c}
\def\Lsymm{L^\symm}
\def\PU{P_\Us}
\def\Chi{X}
\def\Lmod{L_1}

\def\Us{\mathcal{U}}
\def\As{\mathcal{A}}
\def\Cs{\mathcal{C}}

\def\Rper{\R^{N}}

\def\Ts{\mathcal{T}}

\DeclareMathOperator{\sinc}{sinc}
\DeclareMathOperator{\cond}{cond}
\DeclareMathOperator{\lspan}{span}
\DeclareMathOperator{\range}{rg}
\DeclareMathOperator{\krn}{ker}

\begin{document}

\title[The Spectrum of the QCF Operator]{The Spectrum of the
Force-Based Quasicontinuum Operator for a Homogeneous Periodic Chain}

\author{M. Dobson} \address{M. Dobson \\ CERMICS - Ecole des Ponts
ParisTech\\
6 et 8 avenue Blaise Pascal\\
Cit\'e Descartes - Champs sur Marne\\
77455 Marne la Vall\'ee Cedex 2\\ France}
\email{dobsonm@cermics.enpc.fr}

\author{C. Ortner}
\address{C. Ortner\\ Mathematical Institute\\
24-29 St Giles' \\ Oxford OX1 3LB \\ UK}
\email{ortner@maths.ox.ac.uk}

\author{A. V. Shapeev}
\address{A. V. Shapeev\\ Section of Mathematics, Swiss Federal Institute of Technology (EPFL), Station 8, CH-1015, Lausanne, Switzerland}
\email{alexander.shapeev@epfl.ch}

\date{\today}

\thanks{This work was supported in part by the EPSRC Critical Mass Programme
``New Frontiers in the Mathematics of Solids'' and the NSF
Mathematical Sciences Postdoctoral Research Fellowship.  This
project was initiated by the authors at the OxMOS Workshop on
Atomistic Models of Solids.}



\begin{abstract}
  We show under general conditions that the linearized force-based
  quasicontinuum (QCF) operator has a positive spectrum, which is
  identical to the spectrum of the quasinonlocal quasicontinuum (QNL)
  operator in the case of second-neighbour interactions. Moreover, we
  establish a bound on the condition number of a matrix of
  eigenvectors that is uniform in the number of atoms and the size of
  the atomistic region.  These results establish the validity of and
  improve upon recent conjectures (\cite[Conjecture
  2]{Dobson:qcf.stab} and \cite[Conjecture 8]{Dobson:qcf.iter}) which
  were based on numerical experiments.

  As immediate consequences of our results we obtain rigorous
  estimates for convergence rates of (preconditioned) GMRES
  algorithms, as well as a new stability estimate for the QCF method.
\end{abstract}

\maketitle

\section{Introduction}
\label{sec:intro}

Quasicontinuum methods are a prototypical class of multiscale models
that directly couple multiple modeling regions to reduce the
computational complexity of modelling large atomistic systems. These
methods are useful for computing the interaction of localized material
defects such as crack tips or dislocations with long-range elastic
fields of a crystalline material. The force-based quasicontinuum (QCF)
method \cite{curt03,Dobson:2008a,Shenoy:1999a} partitions the material
into two disjoint regions, the atomistic region and the continuum
region.  It assigns forces to the degrees of freedom within each
region using only the respective model, be it atomistic or
continuum. This simplifies the formulation of the method as no special
interaction rules are needed near the atomistic-continuum interface.
The simplicity of mixing forces combined with the lack of spurious
interface forces (so-called ``ghost forces'') make the force-based
method a popular approach, and this technique is widely applied in 
the multiscale literature~\cite{badi08,hybrid_review,curt03,Kohlhoff:1989,Luan:2006,Shenoy:1999a,cadd}.

A potential drawback of the QCF method is that it does not derive from
an energy (as it generally produces a non-conservative field). While
the practical implications of this fact are still under investigation,
it is already clear that the analysis of the QCF method poses
formidable challenges. A series of recent articles has been devoted to
its study~\cite{Dobson:qcf.iter,Dobson:qcf.stab, Dobson:qcf2}. For
example, it was shown in \cite{Dobson:qcf.stab, Dobson:qcf2} that the
linearized QCF operator is not positive definite, and that it is not
uniformly stable (in the number of atoms and the size of the atomistic
region) in discrete variants of most Sobolev spaces.

However, numerical experiments in
\cite{Dobson:qcf.iter,Dobson:qcf.stab} showed some unexpected spectral
properties. Conjecture 2 in \cite{Dobson:qcf.stab} states that the
spectrum of $\ell^2$-eigenvalues of the linearized QCF operator is
identical to that of the operator associated with the quasinonlocal QC
method~\cite{Shimokawa:2004}. This is particularly surprising since
the quasinonlocal QC method is energy-based, and thus indicates that
the QCF operator is diagonalizable and that its spectrum is
real. Conjecture 8 in \cite{Dobson:qcf.iter} states that the condition
number of a matrix of eigenvectors of the QCF operator grows at most
logarithmically. This is an important fact for understanding the
solution of the QCF equilibrium equations by the GMRES method.

The purpose of the present paper is to provide rigorous proofs for
these numerical observations. We define the QCF method and introduce
the necessary notation in Section \ref{sec:prelims}. In Section
\ref{sec:n2} we establish all results for the case of next-nearest
neighbour interactions as in the numerical experiments in
\cite{Dobson:qcf.iter,Dobson:qcf.stab}. Then, in Section \ref{sec:fr}
we extend the results to the more technical case of finite-range
interactions. In the case of finite-range interactions we cannot make
the comparison between the spectra of the QCF and QNL
operators. Instead, we prove that the spectrum of the QCF operator
lies between the spectrum of the atomistic operator and the spectrum
of the continuum operator (Theorem \ref{th:fr:evals}). We note,
moreover, that we were able to construct a matrix of eigenvectors for
the QCF operator whose condition number is bounded uniformly in the
number of atoms and the size of the atomistic region. This result is
in fact stronger than the conjectures made in
\cite{Dobson:qcf.iter}. Finally, in Section \ref{sec:prec}, we analyze
variants of preconditioned QCF operators to obtain rigorous
convergence rates for preconditioned GMRES methods as well as a new
stability estimate.

\section{Formulation of the QCF Method}
\label{sec:prelims}
For the sake of brevity, we will keep the introduction to the
atomistic model and the various flavours of quasicontinuum
approximations to an absolute minimum. We refer
to~\cite{Dobson:2008a,Dobson:qcf.iter,Dobson:qcf.stab,Dobson:qcf2,E:2006,Miller:2003a,emingyang,Ortiz:1995a}
for detailed discussions. Note, in particular, that we have left out
the usual rescaling factor $\eps$. This reduces the complexity of the
notation and is justified since in this paper we are primarily
concerned with algebraic aspects of quasicontinuum operators.

\subsection{Notation for difference operators}
In this section, we summarize the notation and certain elementary results
for some standard finite difference operators with periodic boundary
conditions.

\subsubsection{Periodic domains.}
We identify $\R^{N}$ with periodic infinite sequences as follows:
\begin{displaymath}
\Rper = \big\{ u \in \R^{\Z} : u_{\ell+N} = u_\ell \text{ for all }
\ell \in \Z \big\}.
\end{displaymath}
The $\ell^2$-inner product on $\Rper$, and its associated norm, are
defined as
\begin{displaymath}
\< u, v \> = u^T v = \sum_{\ell = 1}^N u_\ell v_\ell,
\quad \text{and} \quad \|u\| = \sqrt{\<u,u\>}.
\end{displaymath}
We will frequently use a subspace $\Us \subset \Rper$ of mean zero
functions,
\begin{equation*}
\Us = \big\{ u \in \Rper : \< u, e \> = 0 \big\},
\end{equation*}
where $e = (1)_{\ell \in \Z} \in \Rper$.

The orthogonal projection onto $\Us$ is denoted $\PU : \Rper \to
\Rper$,
\begin{equation*}
(\PU u)_\ell = u_\ell - \frac{1}{N} \sum_{k = 1}^N u_k,
\end{equation*}
or, in matrix notation,
\begin{equation}
\label{eq:Pu}
\PU = I - \smfrac{1}{N} e \otimes e.
\end{equation}

\subsubsection{The backward difference operator} 
The difference operator $D: \Rper \to \Rper$ is defined by
\begin{displaymath}
Du_\ell = (Du)_\ell = u_\ell - u_{\ell-1}.
\end{displaymath}
We note that $\range(D) = \Us$ and $\krn(D) = \lspan\{e\}$ where
$\range$ denotes the range and $\ker$ denotes the kernel of an
operator. We also remark that, here and throughout, unless
specifically stated otherwise, we will not distinguish between an
operator and its associated matrix representation in $\R^{N \times
 N}$.

\subsubsection{The discrete Laplace operator}
The second generic operator that we will encounter is the negative
Laplace operator $L : \Rper \to \Rper$,
\begin{displaymath}
Lu_\ell = (Lu)_\ell = -u_{\ell-1} + 2 u_\ell - u_{\ell+1}.
\end{displaymath}
As for the difference operator, $\range(L) = \Us$ and $\krn(L) = \lspan\{e\}$. 

Using summation by parts, we obtain
\begin{displaymath}
\< L u, v \> = \< D u, Dv \>,
\end{displaymath}
which implies that $L = D^T D$ and hence $L=L^T$. Since $Le = 0$, we
have the identities
\begin{equation}
\label{eq:LPu}
L \PU = \PU L = L.
\end{equation}
We also note that $\| L \| \leq 4$, and that this bound is attained
for even $N$, as well as in
the limit $N \to \infty$.

Since $L$ is singular, we also define the modified negative Laplace
operator
\begin{equation}
\label{eq:Lmod}
\Lmod = L + e \otimes e = L + (I-\PU),
\end{equation}
so that $\Lmod u = L u$ if $u \in \Us$ and $\Lmod e = e.$
This operator is invertible and satisfies
\begin{displaymath}
\Lmod^{-1} L = L \Lmod^{-1} = \PU.
\end{displaymath}

\subsubsection{The translation operator} 
The translation operator $T: \Rper \to \Rper$ is defined by
\begin{equation}
\label{eq:defn_T}
Tu_\ell = (Tu)_\ell = u_{\ell+1}.
\end{equation}
$T$ is an orthogonal operator, i.e., $T^T T=I$ and its eigenbasis can
be written explicitly as
\begin{equation}
\label{eq:T-eigenbasis}
T w_k = \lambda_k w_k
, \quad 
\lambda_k = e^{\frac{i 2 \pi k}{N}}
, \quad 
(w_k)_\ell = e^{\frac{i 2 \pi k \ell}{N}}
\quad \text{for } 1\le k\le N.
\end{equation}
The eigenvalues of $T$ are located on the unit circle $\Ts :=
\{t\in\C:\ |t|=1\}$ and, in the limit $N\to\infty$, are dense in
$\Ts$.

We remark that we can write the difference operator $D$ and the
negative Laplace operator as Laurent polynomials in $T$: $D = p_D(T)$
and $L = p_L(t)$ where
\begin{displaymath}
p_D(t) = (1 - t^{-1}), \quad \text{and} \quad p_L(t) = (-t + 2 - t^{-1}).
\end{displaymath}
In general, if $p(t)$ is a polynomial, then the spectrum of the
operator $p(T)$ is $\{p(\lambda_k):\,1\le k\le N\}$, and the
eigenvectors are the same as for $T$. Since $T$ is a normal operator, all
polynomials $p(T)$ are also normal. 

Finally, we note that $Te = e$, which implies that $T$ or any
polynomial of $T$ commutes with $e \otimes e$. In particular, this
implies that all polynomials in $T$ (e.g., $L$, $D$) and the operators
$L_1$ and $\PU$ commute.

\subsection{The linearized atomistic operators}
\label{sec:atomistic_model}
We consider an atomistic model problem with periodic boundary
conditions. We let $\Us$ be the set of admissible displacements of an
$N$-periodic chain: the set of all $N$-periodic displacements with
mean zero. The latter condition is necessary to ensure that the
systems of equations that we consider are well posed. If $F > 0$ is a
fixed {\em macroscopic strain}, then the energy (per period) of the
atomistic chain subject to a displacement $u \in \Rper$ is given by
\begin{equation*}
\E^\a(u) = \sum^{R}_{r = 1} \sum_{\ell = 1}^N 
 \phi\big(rF + (u_{\ell} - u_{\ell-r})\big)
\end{equation*}
where $\phi \in \CC^2(0, +\infty)$ is a pair interaction potential,
for example, a Lennard--Jones or Morse potential, and $R \in \N$, $R
\geq 2$, can be thought of as a discrete cutoff radius. (Note that,
even though we have defined $\E^\a$ for all $u \in \R^{N}$, only $u
\in \Us$ are admitted in the solution of the minimization problem.)

The Cauchy--Born or local quasicontinuum (QCL) approximation of
$\E^\a$ is the functional
\begin{equation*}
\E^\c(u) = \sum_{\ell = 1}^N W\big(F + (u_\ell - u_{\ell-1}) \big)
= \sum_{\ell = 1}^N W\big(F + Du_\ell \big),
\end{equation*}
where $W$ is the {\em Cauchy--Born stored energy function}, $W(s) =
\sum^R_{r=1} \phi(r s)$. 

Our analysis in the present paper concerns properties of the Hessians
$\La = D^2 \E^\a(0)$ and $\Lc = D^2\E^\c(0)$ and the quasicontinuum
operators that we derive from them. For future reference we write out
$\La$ and $\Lc$ explicitly,
\begin{align}
\label{eq:La_explicit}
(\La u)_\ell =~& \sum^R_{r=1} \phi_{rF}'' 
    (- u_{\ell+r} + 2 u_\ell - u_{\ell-r}), \qquad \text{and} \\
\label{eq:Lc_explicit}
(\Lc u)_\ell =~& \sum^R_{r=1} \phi_{rF}'' 
r^2 (- u_{\ell+1} + 2 u_\ell - u_{\ell-1}) = W_F''~(Lu)_\ell, 
\end{align}
where the constants $\phi_{rF}''$ and $W_F''$ are given by
\begin{equation*}
\phi_{rF}'' = \phi''(rF) \quad
\text{and} \quad W_F'' = W''(F) = \sum^R_{r=1} \phi_{rF}'' r^2.
\end{equation*}

We understand both $\La$ and $\Lc$ as linear operators from $\Rper$ to
$\Rper$, defined by the above formulas, but are primarily interested
in their properties on $\Us$. For example, we note that if $\phi_F'' >
0$ and $W_F'' > 0$, then both are positive definite on $\Us$ and in
particular invertible (see \cite[Eq. (2.2) and
Sec. 2.2]{Dobson:qcf.stab} and \cite[Prop. 1 and
Prop. 2]{Dobson:qce.stab} for the next-nearest neighbour case, and
\cite{HudsOrt:a} for finite range), however, both operators have a
non-trivial kernel that contains $e$.  For the continuum operator,
$\Lc,$ the stability condition $W_F''>0$ is sharp, that is, $\Lc$ is
positive definite if and only if $W_F''>0.$  We work to prove our
spectral results on $\Lqcf$ up to this sharp stability criterion.

\subsection{The force-based quasicontinuum method}
\label{defnqcf}
The force-based quasicontinuum (QCF) method is obtained by mixing the
forces from the atomistic and the continuum model. To this end, we define
atomistic and continuum regions $\As$ and $\Cs$ that satisfy
\begin{equation}
\As \cup \Cs = \{1, \dots, N\} \qquad \text{ and } \qquad
\As \cap \Cs = \emptyset.
\end{equation}
We define the QCF forces
\begin{equation*}
F_\ell^\qcf(u) = \cases{
- \frac{\pp \E^\a(u)}{\pp u_\ell}, & \text{if } \ell \in \As, \\[2mm]
- \frac{\pp \E^\c(u)}{\pp u_\ell}, & \text{if } \ell \in \Cs.
}
\end{equation*}
Linearization of the nonlinear QCF operator $F(u) =
(F_\ell^\qcf(u))_{\ell = 1}^N$ at $u = 0$, yields the {\em linear
QCF operator} (or simply, {\em QCF operator}), $\Lqcf : \Rper \to
\Rper$,
\begin{equation}
\label{eq:Lqcf_explicit}
(\Lqcf u)_{\ell} =  \cases{
(\La u)_\ell, & \text{ if } \ell \in \As, \\
(\Lc u)_\ell, & \text{ if } \ell \in \Cs,
}
\end{equation}
which is the focus of our studies in the present
paper. 

Unfortunately, $\Lqcf$ as defined above, does not map $\Us$ to $\Us$,
hence we will normally consider the projected QCF operator (see
\cite[Sec. 2.3]{Dobson:qcf.stab} for more detail)
\begin{displaymath}
\Lqcfp = \PU \Lqcf.
\end{displaymath}

To conclude this section we represent $\Lqcf$ in a more compact
way. By considering the characteristic function of $\As$,
\begin{displaymath}
\chi_\ell = \cases{
1, & \text{ if } \ell \in \As, \\
0, & \text{ if } \ell \in \Cs,
}
\end{displaymath}
and the associated diagonal operator $\Chi : \Rper \to \Rper$,
\begin{displaymath}
(\Chi u)_\ell = \chi_\ell u_\ell,
\end{displaymath}
we can write $\Lqcf$ in the form
\begin{equation}
\label{eq:qcf_sum}
\Lqcf = [1-\Chi] \Lc + \Chi \La = \Lc + \Chi [\La - \Lc].
\end{equation}

\subsection{The quasinonlocal quasicontinuum method}
\label{sec:defnqnl}
The second atomistic/continuum hybrid scheme that will feature
prominently in our investigations is the {\em quasinonlocal
 quasicontinuum (QNL) method}~\cite{Shimokawa:2004}.  We note that
the QNL method is only defined for second-neighbour interaction range
(i.e., $R = 2$).  Extensions to further neighbours
exist~\cite{E:2006,li10,Shapeev:2010a}, but we only use the version up
to second neighbours in this paper.  The QNL method is conservative
with energy functional
\begin{align*}
\E^\qnl(u) =~& \sum_{\ell \in \As \cup \Cs} \phi\big(F + (u_\ell-u_{\ell-1})\big)
+ \sum_{\ell \in \As} \phi\big(2F + (u_{\ell+1}-u_{\ell-1})\big) \\
&+ \sum_{\ell \in \Cs} \frac12 \Big\{ 
\phi\big(2F+2(u_\ell-u_{\ell-1})\big)
+ \phi\big(2F+2(u_{\ell+1}-u_{\ell})\big) \Big\}.
\end{align*}

The linearized QNL operator is the Hessian of $\E^\qnl$ at $u = 0$,
that is, $\Lqnl = D^2 \E^\qnl(0)$. The operator $\Lqnl : \Rper \to
\Rper$, is most easily written in variational form
\cite[Sec. 3.3]{Dobson:qce.stab},
\begin{equation}
\label{eq:Lqnl_vari}
\<\Lqnl u, v\> = W_F'' \sum_{\ell \in \As \cup \Cs} Du_\ell Dv_\ell
- \phi_{2F}'' \sum_{\ell \in \As} L u_\ell L v_\ell.
\end{equation}
Based on this representation one can show that, if $W_F'' > 0$ and
$\phi_F''>0,$  then $\Lqnl$ is positive definite on $\Us$ (see
\cite[Prop. 3]{Dobson:qce.stab} for the case when $\phi_{2F}'' \leq
0$; the case $\phi_{2F}'' > 0$ is trivial).

\section{The $\ell^2$-Spectrum of the Second-Neighbour $\Lqcf$ Operator}
\label{sec:n2}
In \cite{Dobson:qcf.stab} the invertibility of the QCF operator was
investigated analytically and numerically, and several interesting
numerical observations were left as conjectures. Similar observations
were also used in \cite{Dobson:qcf.iter} to study the performance of
iterative solution methods for the QCF operator. In the present
section we rigorously establish some of the most fundamental of these
conjectures in the next-nearest neighbour case. We will then extend the
results, to finite range interactions in Section \ref{sec:fr}.

\subsection{Similarity of $\Lqcfp$ and $\Lqnl$}

In \cite{Dobson:qcf.iter,Dobson:qcf.stab} it was observed in numerical
experiments that the spectra of $\Lqcfp$ and $\Lqnl$ coincide. In this
section we provide a rigorous proof by explicitly constructing a
similarity transformation between $\Lqcfp$ and $\Lqnl$. The main
ideas, after which the proof will be straightforward, are the
following two new representations of the $\Lqcf$ and $\Lqnl$
operators.

\begin{proposition}
Let $R = 2$, then $\Lqcf$ and $\Lqnl$ have, respectively, the
representations
\begin{align}
\label{eq:qcf_rep}
\Lqcf =~& W_F'' L - \phi_{2F}'' \Chi L^2, \qquad \text{and} \\
\label{eq:qnl_rep}
\Lqnl =~& W_F'' L  - \phi_{2F}'' L \Chi L.
\end{align}
\end{proposition}
\begin{proof}
We begin by noting that, for $R = 2$, the operators $\La$ and $\Lc$
may be written as
\begin{align*}
\La =~& \phi_F'' L + \phi_{2F}'' [4 L - L^2] = W_F'' L - \phi_{2F}'' L^2, \\
\Lc =~& \phi_F'' L + \phi_{2F}'' [4 L] = W_F'' L.
\end{align*}

Using these formulas, the operator $\Lqcf$ (as defined in
\eqref{eq:Lqcf_explicit}) can be written in terms of the atomistic
and the continuum operators~\eqref{eq:qcf_sum} as
\begin{displaymath}
\Lqcf = \Chi \La + [I-\Chi] \Lc 
= \Chi \left[ W_F'' L - \phi_{2F}'' L^2 \right]
+ [I-\Chi] \left[ W_F'' L \right].
\end{displaymath}
From this we immediately obtain \eqref{eq:qcf_rep}.

To rewrite the QNL operator, we note that we can write
\eqref{eq:Lqnl_vari} as
\begin{align*}
\< \Lqnl u, v \> =~& W_F'' \< Du, Dv \> - \phi_{2F}'' \< \Chi Lu, Lv \> \\
=~& W_F'' \< D^T D u, v \> - \phi_{2F}'' \< L \Chi L u, v \>,
\end{align*}
for all $u, v \in \Rper$, and we therefore obtain
\eqref{eq:qnl_rep}.
\end{proof}

\noindent Based on \eqref{eq:qcf_rep} and \eqref{eq:qnl_rep} we will
deduce the similarity of the QCF and QNL operators.  Since $L$ is not
invertible, we introduce the nonsingular operator $\Lmod : \R^{N} \to
\R^{N}$ defined by \eqref{eq:Lmod}.  The following result confirms
Conjecture 2 in \cite{Dobson:qcf.stab} (a related conjecture with
different boundary conditions is Conjecture 6 in
\cite{Dobson:qcf.iter}).

\begin{theorem}[Similarity of $\Lqcfp$ and $\Lqnl$]
\label{thm:Lqcf-Lqnl-similarity}
If $R = 2$ then the operators $\Lqcfp$ and $\Lqnl$ are similar, with
similarity transformation $\Lmod$ defined by \eqref{eq:Lmod}:
\begin{displaymath}
\Lqcfp = \Lmod^{-1} \Lqnl \Lmod.
\end{displaymath}
In particular, the spectra of $\Lqcfp$ and $\Lqnl$
coincide. 
\end{theorem}
\begin{proof}
Using formulas \eqref{eq:LPu}, \eqref{eq:qcf_rep}, and
\eqref{eq:qnl_rep}, a straightforward computation yields the desired
identity:
\begin{align*}
\Lmod \Lqcfp =~& \Lmod \PU \Lqcf 
= L\Lqcf  \\
=~& W_F'' LL - \phi_{2F}'' L\Chi L^2 \\
=~& [W_F'' L - \phi_{2F}'' L \Chi L] \Lmod = \Lqnl \Lmod. \qedhere
\end{align*}
\end{proof}

\subsection{Condition number of the $\ell^2$ eigenbasis}
\label{sec:cond_ell2}
Since $\Lqnl$ is self-adjoint, there exists an orthonormal matrix
$\Vqnl \in \R^{N \times N}$, and a diagonal matrix $\Lambda$
containing the eigenvalues of $\Lqnl$, such that
\begin{equation}
\label{eq:qnleig}
\Lqnl \Vqnl = \Vqnl \Lambda.
\end{equation}
Note in particular that $\Lambda$ also contains the zero eigenvalue.
Since the operators $\Lqnl$ and $\Lqcfp$ are similar, there exists also
an invertible operator $\Vqcf \in \R^{N \times N}$ such that
\begin{displaymath}
\Lqcfp \Vqcf = \Vqcf \Lambda.
\end{displaymath}

As suggested by Theorem \ref{thm:Lqcf-Lqnl-similarity}, a possible
choice for the eigenvectors is $L_1^{-1} \Vqnl$, since in that case we
have
\begin{displaymath}
\Lqcfp\,L_1^{-1} \Vqnl = L_1^{-1} \Lqnl \Vqnl = L_1^{-1} \Vqnl\,\Lambda.
\end{displaymath}
However, $\cond(L_1^{-1} \Vqnl) = \cond(L_1^{-1}) = O(N^2)$, which is
much worse than the numerical observations in \cite{Dobson:qcf.stab}
and suggests a poor scaling of the eigenvectors.

To produce a better eigenbasis, it is important to note that the
choice of eigenvectors is not unique even after fixing the ordering as
we are always free to rescale them. This turns out to be a crucial
ingredient in our following construction of an eigenbasis with a
uniformly bounded condition number. The following result is inspired
by Figure 4.2 in \cite{Dobson:qcf.stab} and Conjecture 8 in
\cite{Dobson:qcf.iter}. It does not precisely confirm these, but
establishes a closely related and in fact stronger result.

\begin{theorem}
\label{th:condell2}
Suppose that $R = 2$, then the operator $\Vqcf : \R^{N} \to
\R^{N}$,
\begin{equation}
\label{eq:condell2:Vqcf}
\Vqcf = [W_F'' I - \phi_{2F}'' \PU \Chi L] \Vqnl,
\end{equation}
diagonalizes $\Lqcfp,$ that is, $\Lqcfp \Vqcf = \Vqcf \Lambda$, where
$\Lambda$ is the diagonal matrix of eigenvalues associated with
$\Lqnl$~\eqref{eq:qnleig}.  Moreover, if $W_F'' > 0$ and
$\phi_F''>0$ then $\cond(\Vqcf)$ is bounded above by a
constant that depends on $\phi_{2F}''/W_F''$, but is independent of
$N$ and $\As$.
\end{theorem}

\begin{remark}
\label{rem:rescale_evecs}
The choice of $\Vqcf$ is motivated by the following calculation.
Starting with the similarity result of
Theorem~\ref{thm:Lqcf-Lqnl-similarity}, we derive
\begin{displaymath} 
  \Lqcfp = \Lmod^{-1} \Lqnl \Lmod = \Lmod^{-1} \Vqnl \Lambda
  (\Vqnl)^T \Lmod,
\end{displaymath}
and we scale the eigenvectors by $\Lambda,$ which gives
\begin{align*}    
  \Lmod^{-1} \Vqnl\Lambda
  =~& \Lmod^{-1} \Vqnl \Lambda [\Vqnl]^T \Vqnl  
  = \Lmod^{-1} \Lqnl \Vqnl \\
  =~& \Lmod^{-1} [W_F'' L - \phi_{2F}'' L \Chi L] \Vqnl 
  = [W_F'' \PU - \phi_{2F}'' \PU \Chi L] \Vqnl.
\end{align*}
This is equivalent to the choice of $\Vqcf$ in~\eqref{eq:condell2:Vqcf} when
restricted to $\Us.$
\end{remark}

\begin{proof}[Proof of Theorem \ref{th:condell2}] 
{\it Step 1: Diagonalization. } In a straightforward computation we
obtain
\begin{align*}
\Lqcfp \Vqcf =~& \big[ W_F'' L - \phi_{2F}'' \PU \Chi L^2 \big] \,
\big[ W_F'' I - \phi_{2F}'' \PU \Chi L \big] \Vqnl \\
=~& \big[W_F'' I - \phi_{2F}'' \PU \Chi L \big] \,
\big[ W_F'' L - \phi_{2F}'' L \Chi L \big] \Vqnl \\
=~& \big[W_F'' I - \phi_{2F}'' \PU \Chi L \big] \,
\Vqnl \Lambda \\
=~& \Vqcf \Lambda.
\end{align*}

{\it Step 2: Estimating $\cond(\Vqcf)$. }  We now assume that 
$W_F'' > 0$ and $\phi''_F > 0.$  To estimate $\cond(\Vqcf)$
we can ignore the positive constant multiple $W_F''$ as well as the
orthonormal matrix $\Vqnl$, that is, we have
\begin{align}
\label{eq:cond_Vqcf_vs_A}
&\cond(\Vqcf) = \cond(A) = \|A\|\,\|A^{-1}\|, \\
\notag 
& \qquad \text{where} \qquad 
 A = I - \alpha \PU \Chi L,
\end{align}
with constant $\alpha=\frac{\phi_{2F}''}{W_F''}$, and the convention
$\|A^{-1}\| = +\infty$ if $A$ is not invertible. We note that the
condition $\phi_F'', W_F'' > 0$ implies that $\alpha < 1/4$.

Elementary estimates give the upper bound
\begin{equation}
\label{eq:condell2:upperbnd}
\| A \| \leq 1 + |\alpha| \|\PU\|\|\Chi\| \|L\|
= 1 + 4 |\alpha|.
\end{equation}
We similarly get the lower bound 
\begin{equation}
\label{eq:condell2:lowerbnd}
\| A u \| \geq 1 - |\alpha| \|\PU\|\|\Chi\| \|L\|
= 1 - 4 |\alpha| \qquad \text{ for all } \| u\| = 1,
\end{equation}
which gives an estimate for $\|A^{-1}\|$ whenever $W''_F + 4
\phi''_{2F} >0.$ In the following we prove a bound for $\|A^{-1}\|$
that holds whenever $W''_F >0,$ that is, up to the sharp stability 
limit, which is a more
involved result.

To estimate $\|A^{-1}\|$ we use the facts that (i) $\|A^{-1} \| =
\|A^{-T}\|$, and (ii) if
\begin{displaymath}
\| A^T u \| \geq \gamma_0 \|u\| \qquad \forall u \in \R^{N},
\end{displaymath}
for some constant $\gamma_0 > 0$, then $A^T$ is invertible and
$\|A^{-T}\| \leq 1 / \gamma_0$. In Lemma \ref{th:condell2_lowerbd} we
establish precisely this fact, assuming that $\alpha < 1/4$, with a
constant $\gamma_0$ that depends only on $\alpha$ but not on $N$ or
$\As$.
\end{proof}

\noindent A generalization of the following technical lemma used in
the previous proof will also be required in the finite range
interaction case. It follows from Lemma \ref{th:fr:condell2_lowerbd}
by choosing $Z = I$.

\begin{lemma}
\label{th:condell2_lowerbd}
Let $\alpha < 1/4$, then there exists a constant $\gamma_0 > 0$, which
depends on $\alpha$ but is independent of $N$ and of $\As$, such that
\begin{displaymath}
\big\| [I - \alpha \PU \Chi L]^T u \big\| \geq \gamma_0 \|u\| 
\qquad \forall u \in \R^{N}.
\end{displaymath}
\end{lemma}

\begin{remark}
From the proof of Lemma \ref{th:fr:condell2_lowerbd} we see that, in
the case $\phi_{2F}'' \leq 0$, the constant $\gamma_0$ is explicitly
given by
\begin{equation}
  \label{eq:condell2:explicit_gamma0}
  \gamma_0^2 = 1 + 8 \alpha^2 - 4 \sqrt{ \alpha^2
    + 4 \alpha^4},
\end{equation}
where $\alpha = \phi_F'' / W_F''$, and the resulting condition
number estimate by
\begin{displaymath}
  \cond(\Vqcf; \Us)^2 \leq \frac{(1 + 4 |\alpha|)^2}{
    1 + 8 \alpha^2 - 4 \sqrt{ \alpha^2 + 4 \alpha^4}} =: c(\alpha)^2.
\end{displaymath}
If $\phi_{2F}''$ is moderate but $W_F'' \to 0$, then $\alpha \to
-\infty$. A brief calculation shows that in this limit $c(\alpha)$
behaves asymptotically like
\begin{displaymath}
  c(\alpha) \sim 2^{5/2} \alpha^2 + O(|\alpha|^{3/2}) \qquad \text{as} \quad
  \alpha \to \infty. \qedhere
\end{displaymath}
\end{remark}

\section{Finite Range Interactions}
\label{sec:fr}
In this section we extend the results of Section \ref{sec:n2} to the
case of finite range interactions (i.e., with arbitrary finite $R$).
We begin by stating a simplified formulation of our main results.  We
note, however, that our actual results are more general than the
following theorem. In particular, we can replace the assumption
$\phi_{rF}'' \leq 0$, $r \geq 2$, by a more general condition. This
theorem will be proved in Section \ref{sec:fr:nonpos}.

\begin{theorem}
\label{th:fr:mainthm_simple}
Suppose that $\phi_{rF}'' \leq 0$ for $2 \leq r \leq R$ and that
$\phi_{RF}'' \neq 0$.

(i) There exists an operator $\Vqcf : \Rper \to
\Rper$, which diagonalizes $\Lqcfp$, that is,
\begin{displaymath}
  \Lqcfp \Vqcf = \Vqcf \Lambda,
\end{displaymath}
where $\Lambda$ is a diagonal real matrix of eigenvalues
$(\lambda_j)_{j = 1}^{N}$.

(ii) If $W_F'' > 0$ then $\Vqcf$ is invertible and $\cond(\Vqcf)$ is
bounded above by a constant that depends on the coefficients
$\phi_{rF}''$, $r = 1, \dots, R$, but is independent of $N$ and $\As$.

(iii) If $W_F'' > 0$ and if the eigenvalues are ordered, then
\begin{displaymath}
  \lambda_j^\c \leq \lambda_j \leq \lambda_j^\a,
\end{displaymath}
where $(\lambda_j^\c)_{j=1}^{N}$ and $(\lambda_j^\a)_{j=1}^{N}$
denote the ordered eigenvalues of, respectively, $\Lc$ and $\La$. In
particular, we have the bounds
\begin{displaymath}
  \lambda_1 = 0, \quad \text{and} \quad
  4 W_F'' \sin^2\big(\smfrac{\pi}{N}\big) \leq \lambda_j \leq 4 \phi_F''
  \quad \text{for } j = 2, \dots, N.
\end{displaymath}
\end{theorem}

\subsection{Symmetrization of $\Lqcf$}
\label{sec:fr:symm}
We recall from \eqref{eq:qcf_sum} the definition of the finite range
QCF operators,
\begin{displaymath}
\Lqcf = \Lc + \Chi [\La - \Lc],
\quad \text{and} \quad
\Lqcfp = \Lc + \PU \Chi [\La - \Lc],
\end{displaymath}
where $\La$ and $\Lc$ are, respectively, the atomistic and the
continuum operators defined in \eqref{eq:La_explicit} and
\eqref{eq:Lc_explicit}. In terms of the translation operator defined
in \eqref{eq:defn_T}, we can express $\La$ and $\Lc$ as
\begin{equation}
\label{eq:La-definition}
\Lc = \sum_{r=1}^{R} r^2 \phi''_{rF} \big[-T + 2I - T^{-1} \big],
\quad \text{and} \quad 
\La = \sum_{r=1}^{R} \phi''_{rF} \big[-T^r + 2I - T^{-r} \big].
\end{equation}

To show that $\Lqcfp$ is diagonalizable, we aim to construct a matrix
$Y$ such that $Y \Lqcf Y^{-1}$ is symmetric. We first note that we can
write the difference $\La - \Lc$ as a Laurent polynomial of the
translation operator $T$:
\begin{align*}
\La-\Lc 
=~& \sum_{r=2}^{R} \phi''_{rF} \big[(-T^r+2I-T^{-r}) - r^2 (-T+2 I-T^{-1})\big] \\
=~& \sum_{r=2}^{R} \phi''_{rF} \big[(T^r-I)(T^{-r}-I) - r^2 (T-I) (T^{-1}-I)\big].
\end{align*}
Thus, if we define the Laurent polynomial
\begin{equation}
\label{eq:b-definition}
b(t) = \sum_{r=2}^{R} \phi''_{rF} \big[(t^r-1)(t^{-r}-1)-r^2 (t-1) (t^{-1}-1)\big]
\quad \text{for } t \in \C \setminus \{0\},
\end{equation}
then we obtain $\La - \Lc = b(T)$. The crucial observation is the
following: If we can factorize $b(t)$ as $b(t) = p(t) p(1/t)$ then we
can choose $Y = p(T)$ to obtain
\begin{equation}
\label{eq:fr:YYt}
\La-\Lc = b(T) = p(T) p(T^{-1}) = p(T) p(T^{T}) = Y Y^T = Y^T Y,
\end{equation}
and we immediately obtain
\begin{equation}
\label{eq:fr:symm_Lqcf}
Y \Lqcf
= Y \Lc + Y \Chi \big[Y^T Y\big]
= \big[\Lc + Y \Chi Y^T\big] Y.
\end{equation}
Here, we have used the fact that all operators that are polynomials in $T$
commute.  A similar result holds also for $\Lqcfp$, though this would
require further preparation. If $Y$ were invertible (this will turn
out to be false), then this would imply that $\Lqcf$ is similar to a
symmetric, hence, normal matrix, and is therefore diagonalizable.

The desired polynomial factorization result is essentially the
Riesz--F\`{e}jer factorization lemma \cite[Sec. 53]{RieszNagy}, which
we state and prove in a slightly more general form.

\begin{lemma}
\label{th:symmetric-product}
Let $b(t)$ be a Laurent polynomial with real coefficients such that
$b(t) = b(1/t)$. Then there exists a polynomial $p(t)$ such that
$b(t) = p(t) p(1/t)$. 

If, in addition, $b(t) \ge 0$ for all $t \in \Ts := \{ t \in \C :
|t| = 1\}$, then $p(t)$ can be chosen to have real coefficients.
\end{lemma}

\begin{remark}
\label{rem:fr:neg_b}
In Lemma \ref{th:symmetric-product}, if $b(t) \leq 0$ on $\Ts$, then
it can be factorized as $b(t) = - p(t) p(1/t)$, where $p(t)$ has real
coefficients.
\end{remark}

\begin{proof}[Proof of Lemma \ref{th:symmetric-product}]
 Let $\alpha t^{-R}$ be the leading term in $b(t)$ with $\alpha \neq
 0$, then $a(t) := t^R b(t)$ is a polynomial with $a(0) \neq
 0$. Moreover, $a$ and $b$ share the same roots, which we collect
 into a set $\Lambda$, so that
\begin{displaymath}
  b(t) = \alpha t^{-R} \prod_{\lambda\in\Lambda} (t-\lambda)^{m(\lambda)},
\end{displaymath}
where $m(\lambda)$ denotes the multiplicity of $\lambda$.

Next, we define the auxiliary polynomial
\begin{displaymath}
  p_1(t) = \prod_{\lambda\in\Lambda, |\lambda|<1} (t-\lambda)^{m(\lambda)}.
\end{displaymath}
Since $b(t)$ has real coefficients it follows that $\lambda \in
\Lambda$ if and only if $\bar\lambda \in \Lambda$, and $m(\lambda) =
m(\bar\lambda)$, and hence $p_1(t)$ also has real coefficients. We also note that
\begin{displaymath}
  p_1(1/t) = \prod_{\lambda \in \Lambda, |\lambda|< 1} (- \lambda / t)^{m(\lambda)}
  \big(t - 1/\lambda\big)^{m(\lambda)}.
\end{displaymath}
Since $b(t) = b(1/t)$, for each root $\lambda$ with $|\lambda| > 1$,
$1/\lambda$ is also a root, with the same multiplicity, and
therefore we have
\begin{displaymath}
  b(t) = p_1(t) \, p_1(1/t) \, b_1(t),
\end{displaymath}
where
\begin{displaymath}
  b_1(t) = \alpha_1 t^{-k_1} \prod_{\lambda \in \Lambda, |\lambda| = 1} 
  (t - \lambda)^{m(\lambda)}
\end{displaymath}
for some constants $\alpha_1 \in \R$ and $k_1 \in \N$.  By
construction, $b_1$ has real coefficients and $b_1(t) = b_1(1/t)$.

We will later assume that $b$ has no roots on $\Ts$ except for $t =
1$, and for that case, the proof would be complete. To establish the
lemma in its full generality, we now distinguish two cases.

{\it Case 1: $b(t)\ge 0$ on $\Ts$. } For $t \in \Ts$, we have $1/t =
\bar t$, and hence, upon reordering the product in $p_1(1/t)$,
\begin{displaymath}
  p_1(t)p_1(1/t) = \prod_{\lambda \in \Lambda, |\lambda| < 1} \Big[
  (t - \lambda)^{m(\lambda)} ( \bar t - \bar\lambda)^{m(\lambda)} \Big] > 0
  \quad \text{for all } t \in \Ts.
\end{displaymath}
Thus, if $b(t) \geq 0$ on $\Ts$, then we also have $b_1(t) \geq 0$
on $\Ts$. This implies that the multiplicity $m(\lambda)$ of all
roots $\lambda \in \Ts \cap \Lambda$ is even. We can therefore define
\begin{displaymath}
  p_2(t) = \prod_{\lambda \in \Lambda, |\lambda| = 1} (t - \lambda)^{m(\lambda)/2},
\end{displaymath}
and argue similarly as above, to obtain that $p_2(t)$ has real
coefficients and that
\begin{displaymath}
  b_1(t) = \alpha_2 t^{-k_2} p_2(t) p_2(1/t) =: b_2(t) p_2(t) p_2(1/t),
\end{displaymath}
for some constants $\alpha_2 \in \R$ and $k_2 \in \N$, and $b_2(t) =
\alpha_2 t^{-k_2}$. We can again deduce that $b_2(t) = b_2(1/t)$ and
that $b_2(t) \geq 0$ on $\Ts$, which implies that $k_2 = 0$ and
$\alpha_2 \geq 0$. Thus, we obtain $p(t)$ by defining
\begin{displaymath}
  p(t) = \sqrt{\alpha_2} p_1(t) p_2(t).
\end{displaymath}

{\it Case 2: $b(t) \not \ge 0$ on $\Ts$. } In this case the
multiplicity of roots $\lambda \in \Ts$ may be odd, and therefore we
define
\begin{displaymath}
  p_2(t) = \prod_{\lambda\in\Lambda, |\lambda|=1, {\rm Im}\lambda>0} (t-\lambda)^{m(\lambda)}.
\end{displaymath}
Thus, $p_2(t)$ contains all roots of $b_1(t)$ with positive
imaginary part and $p_2(1/t)$ contains all roots of $b_1(t)$ with
negative imaginary part. We are only left to find any roots at $\pm
1$. Since $b_1(t) = b_1(1/t)$ it follows that these roots must be
even (possibly $m(- 1) = 0$), and hence we can define
\begin{displaymath}
  p_3(t) = (t-1)^{m(1)/2} (t+1)^{m(-1)/2},
\end{displaymath}
to obtain
\begin{displaymath}
  b_1(t) = \alpha_2 t^{k_2} p_2(t) p_2(1/t) p_3(t) p_3(1/t)
\end{displaymath}
for some constants $\alpha_2 \in \R$ and $k_2 \in \Z$. Arguing as
before we find that $\alpha_2 > 0$ and $k_2 = 0$, and hence we
obtain the result if we define
\begin{displaymath}
  p(t) = \sqrt{\alpha_2} p_1(t) p_2(t) p_3(t).
  \qedhere
\end{displaymath}
\end{proof}

For any Laurent polynomial $b(t)$ we call a polynomial $p(t)$ that
satisfies $b(t) = p(t) p(1/t)$ a GRF-factorization (generalized
Riesz--F\`{e}jer factorization) of $b(t)$.

Since the Laurent polynomial $b(t)$ defined in \eqref{eq:b-definition}
clearly satisfies the condition $b(t) = b(1/t)$, we have now indeed
established \eqref{eq:fr:YYt} and \eqref{eq:fr:symm_Lqcf}. Thus, if
$Y$ were invertible, then we could deduce that $\Lqcf$ (and similarly
$\Lqcfp$) is diagonalizable. Unfortunately, as follows from the next
lemma, $Y$ is always singular.

\begin{lemma}
\label{th:fr:defn_p1}
Let $p(t)$ be any GRF-factorization of the Laurent polynomial $b(t)$
defined in \eqref{eq:b-definition}, then there exists a polynomial
$p_1(t)$ such that $p(t) = (t-1)^2 p_1(t)$.  In particular, the
operator $Y = p(T)$ can be rewritten as
\begin{displaymath}
  Y = L Y_1 \qquad \text{where} \quad Y_1 = - T p_1(T).
\end{displaymath}
\end{lemma}
\begin{proof}
It can be immediately seen from the definition of $b(t)$ that $b(1)
= 0$ and, since $b(1) = p(1)^2$, it follows that $p(1) = 0$. Next,
dividing $b(t)$ by $(t-1)(t^{-1} - 1)$ we obtain
\begin{align*}
  \frac{p(t)p(1/t)}{(t-1)(t^{-1}-1)} =~& \frac{b(t)}{(t-1)(t^{-1} - 1)} \\
  =~& \sum_{r = 2}^{R} \phi_{rF}'' \big[ (1+t+\dots+t^{r-1})(1 + t^{-1}
  + \dots + t^{-r+1}) - r^2 \big],
\end{align*}
and hence
\begin{displaymath}
  \frac{p(t)p(1/t)}{(t-1)(t^{-1}-1)}\bigg|_{t = 1} = 0.
\end{displaymath}
Thus, we see that $1$ is also a (multiple) root of
$\frac{p(t)p(1/t)}{(t-1)(t^{-1}-1)}$, and we can conclude that $p(t)
= (t-1)^2 p_1(t)$ for some polynomial $p_1(t)$.

Upon rewriting the representation $p(t) = (t-1)^2 p_1(t)$ as
\begin{displaymath}
  p(t) = \big[(t-1)(t^{-1}-1)\big]\big[-t p_1(t)\big]
  = \big[- t + 2 - t^{-1} \big]\big[-t p_1(t)\big]
\end{displaymath}
we obtain $Y = p(t) = L Y_1$.
\end{proof}

Thus, as a consequence of Lemma \ref{th:fr:defn_p1}, we see that $Y$
is always singular.  Nevertheless, in order to construct a similarity
transformation to diagonalize $\Lqcfp$, we can make a similar
modification as in the next-nearest neighbour case, and simply replace
$L$ by its invertible variant $\Lmod$ defined in \eqref{eq:Lmod}. Of
course this still leaves us to verify that $Y_1$ is invertible, for
which we will introduce conditions in Sections
\ref{sec:lr:cond_bounds} and \ref{sec:fr:nonpos}.

\begin{proposition}
\label{th:fr:diag_Lqcf_Lqcfp}
Let $p(t)$ be any GRF-factorization of $b(t)$, and let
$Y = p(T) = L Y_1$, as in Lemma \ref{th:fr:defn_p1}, then
\begin{align*}
  [\Lmod Y_1] \Lqcfp = \big[ \Lc + Y X Y^T \big] [\Lmod Y_1].
\end{align*}
In particular, if $Y_1$ is invertible then $\Lqcfp$ is
diagonalizable.
\end{proposition}
\begin{proof}
We use \eqref{eq:fr:YYt} to represent $\La - \Lc$, the fact that
polynomials of $T$ and $\Lmod$ commute, and the fact that $\PU \Lmod
= \Lmod \PU = L$, to obtain
\begin{align*}
  [\Lmod Y_1] \Lqcfp =~& [\Lmod Y_1] \Lc 
  + Y_1 [\Lmod \PU] \Chi L Y_1 Y_1^T L \\
 =~& \Lc [\Lmod Y_1] + Y_1 L \Chi L Y_1^T \PU [\Lmod Y_1] \\
 =~& [\Lc + Y_1 L \Chi L Y_1^T ] \, [\Lmod Y_1].   \qedhere
\end{align*}
\end{proof}

\subsection{Condition number of the $\ell^2$ eigenbasis}
\label{sec:lr:cond_bounds}
To investigate the invertibility of $Y_1$ we will study the Laurent
polynomial
\begin{equation}
\label{eq:fr:defn_b1}
b_1(t) = \frac{b(t)}{\big[ (t-1)(t^{-1}-1) \big]^2}
\end{equation}
on the unit circle $\Ts$. It will quickly become apparent that
$b_1(t)$ needs to be bounded away from zero on $\Ts$ in order to
obtain invertibility of $Y_1$ and bounds on the inverse that are
uniform in $N$ and $\As$. Consequently, we focus on the case where
$b(t)$ does not change sign on $\Ts$ (note that $b(t)$ is always real
on $\Ts$). To show that this is a reasonable assumption we will, in
Section \ref{sec:fr:nonpos}, study the case of non-positive
coefficients $\phi_{rF}'' \leq 0$ for $r = 2, \dots, R$, which is of
particular interest in applications as this condition is satisfied for
most practical interaction potentials. We will show that one can
obtain bounds of the following type: {\it there exist positive
 constants $\beta_0, \beta_1$ such that}
\begin{equation}
\label{eq:fr:b1_bound_ass}
\beta_0^2 \leq |b_1(t)| \leq \beta_1^2 \qquad \forall t \in \Ts.
\end{equation}
Clearly, since $b(t)$ is real on $\Ts$, it is necessary that $b_1(t)$
does not change sign on $\Ts$, that is, either $b_1(t) > 0$ or $b_1(t)
< 0$. If $b_1(t) > 0$ then the matrix $Y_1$ defined in Lemma
\ref{th:fr:defn_p1} is real, however, if $b_1(t) < 0$ then it is
imaginary. For the following analysis we prefer to construct a real
coordinate transform.

According to Remark \ref{rem:fr:neg_b} and Lemma \ref{th:fr:defn_p1},
we can choose a polynomial $p(t) = (t-1)^2 p_1(t)$ with real
coefficients such that
\begin{displaymath}
b(t) = \sigma p(t) p(1/t),
\end{displaymath}
where $\sigma \in \{+1, -1\}$ is the sign of $b_1(t)$ on $\Ts$. Thus, if we
(re-)define $Y_1$ accordingly as
\begin{equation}
\label{eq:fr:new_Y1}
Y_1 = - T p_1(T)
\end{equation}
then Proposition \ref{th:fr:diag_Lqcf_Lqcfp} implies that
\begin{equation}
\label{eq:fr:diagon_newY1}
[\Lmod Y_1] \Lqcfp = \big[\Lc + \sigma (L Y_1) \Chi (L Y_1)^T\big] 
[\Lmod Y_1].
\end{equation}
Next, we show that $\eqref{eq:fr:defn_b1}$ implies invertibility of
$Y_1$, including explicit bounds. Beforehand, however, we make a brief
remark on the connection of the sign of $b(t)$ and a relationship
between $\La$ and $\Lc$.

\begin{remark}
\label{rem:sign_b}
The assumption $b_1(t) > 0$ on $\Ts$, or, more generally, $b(t) \geq
0$ on $\Ts$, implies that
\begin{displaymath}
  \< \La u, u \> \geq \< \Lc u, u \> \qquad \forall u \in \Rper.
\end{displaymath}
This follows simply from the fact that $\La$ and $\Lc$ share the
same eigenvectors, and that the spectrum of $\La - \Lc$, which is
contained in $b(\Ts)$ is non-negative.

Conversely, if $b(t) \leq 0$ on $\Ts$ then, by the same argument,
\begin{displaymath}
  \< \La u, u \> \leq \< \Lc u, u \> \qquad \forall u \in \Rper.
\end{displaymath}
These observations are interesting in their own right, and will
moreover lead to useful estimates on the eigenvalues of the QCF
operator $\Lqcfp$ in Section \ref{sec:fr:evals}.
\end{remark}

\begin{lemma}
\label{th:fr:bounds_Y1}
Let $Y_1$ be defined as in \eqref{eq:fr:new_Y1} and suppose that
$b(t)$ satisfies \eqref{eq:fr:b1_bound_ass}, then $Y_1$ is
invertible and we have the bounds
\begin{displaymath}
  \|Y_1\| \leq \beta_1, 
  \quad \text{and} \quad
  \|Y_1^{-1} \| \leq 1/ \beta_0.
\end{displaymath}
\end{lemma}
\begin{proof}
 Let $p(t) = (t-1)^2 p_1(t)$ be a GRF-factorization of $\sigma b(t)$
 with real coefficients, then $p_1(t)$ is a GRF-factorization of
 $b_1(t)$ with real coefficients. Hence, for $t \in \Ts$, we obtain
\begin{displaymath}
  b(t) = p_1(t)p_1(\bar t) = |p_1(t)|^2,
\end{displaymath}
which implies that
\begin{equation}
  \label{eq:fr:p1_bounds}
  \beta_0 \leq |t p_1(t)| \leq \beta_1 \qquad \forall t \in \Ts.
\end{equation}

By assumption, $Y_1 = - T p_1(T) =: q_1(T)$. Since $T$ is
unitary, this implies that $Y_1$ has an orthonormal eigenbasis and
that the spectrum of $Y_1$ is contained in $q_1(\Ts)$. This
immediately gives
\begin{displaymath}
  \| Y_1 \| \leq \sup_{t \in \Ts} |t p_1(t)| \leq \beta_1.
\end{displaymath}
Moreover, we have
\begin{displaymath}
  \|Y_1^{-1} \|^{-1} = \inf_{\|u\| = 1} \|Y_1 u\| \geq \inf_{t \in \Ts} |t p_1(t)| 
  \geq \beta_0,
\end{displaymath}
which gives the second bound.
\end{proof}

According to \eqref{eq:fr:diagon_newY1}, and Lemma
\ref{th:fr:bounds_Y1}, we have
\begin{equation}
\label{eq:fr:defn_Lsym}
[\Lmod Y_1] \Lqcfp [\Lmod Y_1]^{-1} = \Lc + \sigma [LY_1] X [L Y_1]^T =: L^{\rm sym},
\end{equation}
that is, $\Lqcf$ and $\Lsymm$ are similar matrices.  Since $\Lsymm$ is real and
symmetric, there exists an orthogonal operator $\Vsymm \in
\R^{N\times N}$ such that
\begin{displaymath}
\Lsymm = \Vsymm  \Lambda [\Vsymm]^T,
\end{displaymath}
where $\Lambda$ is the diagonal matrix of eigenvalues of
$\Lqcfp$, and we obtain
\begin{displaymath}
\Lqcfp = [Y_1^{-1} \Lmod^{-1} \Vsymm] \Lambda [Y_1^{-1} \Lmod^{-1} \Vsymm]^{-1},
\end{displaymath}
that is, the operator $Y_1^{-1} \Lmod^{-1} \Vsymm$ diagonalizes
$\Lqcfp$. As in the nearest neighbour case, the eigenbasis $Y_1^{-1}
\Lmod^{-1} \Vsymm$ is poorly scaled and would lead to an $O(N^2)$
condition number. However, the same argument as in Remark
\ref{rem:rescale_evecs} shows how to rescale the basis to obtain the
following theorem.

\begin{theorem}
\label{th:fr:condbnd_abstract}
Suppose that the Laurent polynomial $b(t)$ defined in
\eqref{eq:b-definition} satisfies \eqref{eq:fr:b1_bound_ass}, and
let $Y_1$ be defined by \eqref{eq:fr:new_Y1}. Then the
operator $\Vqcf \in \R^{N \times N}$,
\begin{displaymath}
  \Vqcf = \big[ W_F'' Y_1^{-1}  + \sigma \PU X Y_1^T L \big] \Vsymm
\end{displaymath}
diagonalizes $\Lqcfp$, that is, $\Lqcfp \Vqcf = \Vqcf \Lambda$,
where $\Lambda$ is a real diagonal matrix of eigenvalues.

Moreover, if $W_F'' > 0$, and if $\sigma \beta_1^2 / W_F'' > -1/4$,
then $\Vqcf$ is invertible and $\cond(\Vqcf)$ is bounded above by a
constant that depends only on $W_F'', \beta_0, \beta_1$, and, in
particular is independent of $N$ and $\As$.
\end{theorem} 
\begin{proof}
 The proof of this result is very similar to the proof of Theorem
 \ref{th:condell2}, and hence we shall be fairly brief. First, we
 note that, by Lemma \ref{th:fr:bounds_Y1}, the matrix $Y_1$ can be
 chosen to be invertible, and hence $\Vqcf$ is
 well-defined. Moreover, we recall that $Y_1$ is a polynomial in $T$
 and hence commutes with all operators that commute with $T$, such as
 other polynomials in $T$, the modified negative Laplace operator
 $\Lmod$, and the projection operator $\PU$.

{\it Step 1: Diagonalization. } As in the computation in the proof
of Theorem \ref{th:condell2}, we obtain
\begin{align*}
  \Lqcfp \Vqcf
  =~& \big[ W_F'' L + \sigma \PU \Chi Y_1 Y_1^T L^2 \big] \,
  \big[ W_F'' Y_1^{-1} + \sigma \PU \Chi Y_1^T L \big] \Vsymm \\
  =~& \big[ W_F'' Y_1^{-1} + \sigma \PU \Chi Y_1^T L \big] \,
  \big[ W_F'' L + \sigma L Y_1 \PU \Chi Y_1^T L \big] \Vsymm \\
  =~& \big[ W_F'' Y_1^{-1} + \sigma \PU \Chi Y_1^T L \big] \,
  \Vsymm \Lambda \\
  =~& \Vqcf \Lambda.
\end{align*}

{\it Step 2: Estimating $\cond(\Vqcf)$. } Suppose now that $W_F'' >
0$. As before, estimating $\|\Vqcf\|$ is straightforward. Using
Lemma \ref{th:fr:bounds_Y1}, we obtain
\begin{equation}
  \label{eq:finite-range:condell2:condell2_upperbd}
  \|\Vqcf\| 
  \le
  W_F'' \|Y_1^{-1}\| + \|\PU \Chi Y_1^T L\|
  \le 
  W_F'' / \beta_0 + 4 \beta_1,
\end{equation}
where $\beta_0$ and $\beta_1$ are the constants from
\eqref{eq:fr:b1_bound_ass}.

To estimate $\Vqcf$ from below, we will use again the fact that
$\|[\Vqcf]^{-1} \| = \|[\Vqcf]^{-T}\| \leq 1/(W_F''
\tilde{\gamma}_0)$ where $\tilde{\gamma}_0 > 0$ satisfies
\begin{equation}
  \label{eq:fr:Vqcf_lower_bnd}
  \| [\Vqcf]^T u \| \geq \tilde{\gamma}_0 \|u\| \qquad \forall u \in \Rper.
\end{equation}
Writing out $(\Vqcf)^T$ in full, we get
\begin{align*}
  [\Vqcf]^T =~& W_F'' Y_1^{-T} [\Vsymm]^T \big[I +
  \smfrac{\sigma}{W_F''} Y_1^T L Y_1 \Chi \PU \big] \\
  =~& W_F'' [\Vsymm]^T Y_1^{-T} \big[I - \alpha Y_1^T L Y_1 \Chi \PU \big],
\end{align*}
where $\alpha = - \sigma/W_F''$. If $\alpha \beta_1^2 < 1/4$, then
we can use Lemma \ref{th:fr:condell2_lowerbd} below to deduce that
\begin{displaymath}
  \big\| \big[I + 
  \smfrac{\sigma}{W_F''} Y_1^T L Y_1 \Chi \PU \big] v \big\|
  \geq \gamma_0 \| v \| \qquad \forall v \in \Rper,
\end{displaymath}
where $\gamma_0$ depends only on $\alpha \beta_1^2 = -\sigma
\beta_1^2/W_F''$, but is independent of $N$ and $\As$. This implies
\eqref{eq:fr:Vqcf_lower_bnd} with
$\tilde{\gamma}_0 = W_F'' \gamma_0 / \beta_1$.

Combining \eqref{eq:finite-range:condell2:condell2_upperbd} and
\eqref{eq:fr:Vqcf_lower_bnd} gives the stated result.
\end{proof}

\begin{lemma}
\label{th:fr:condell2_lowerbd}
Let $A = I - \alpha \PU \Chi Z^T L Z$, where $Z \in \R^{N \times
  N}$ commutes with $\PU$, and where $\alpha \in \R$ satisfies
\begin{displaymath}
  -\infty < \alpha \|Z\|^2 < 1/4,
\end{displaymath}
then there exists a constant $\gamma_0 > 0$, depending only on
$\alpha \|Z\|^2$, such that
\begin{displaymath}
  \| A^T u \| \geq \gamma_0 \|u\|
  \qquad \forall u \in \Rper.
\end{displaymath}
\end{lemma}
\begin{proof}
We decompose $A^T$ into
\begin{align*}
  A^T = I + \alpha Z^T L Z \Chi \PU
  = \big[I-\PU\big] + \PU\big[ I - \alpha Z^T L Z \Chi \PU\big],
\end{align*}
where we have used the fact that $\PU L = L$ and that $\PU$
commutes with $Z$.  Since $\PU$ is an
orthogonal projection we obtain, again using $\PU L = L$,
\begin{equation}
  \label{eq:condell2:20}
  \begin{split}
    \| A^T v \|^2 =~& \|[1-\PU]v\|^2 
    + \| [I - \alpha Z^T L Z \Chi]\PU v\|^2.
  \end{split}
\end{equation}
We will show next that
\begin{equation}
  \label{eq:condell2:21}
  \| [1-\alpha Z^TLZ\Chi] w \|^2 \geq (1-\epsilon) \| w \|^2
  \qquad \forall w \in \Us,
\end{equation}
where $\epsilon \in (0, 1)$ is defined in
\eqref{eq:condell2:optimal_epsilon} and depends only on $\alpha
\|Z\|^2$, but not on $N$ or $\As$. Hence, \eqref{eq:condell2:21}
combined with \eqref{eq:condell2:20}, gives the desired result
\begin{displaymath}
  \| A^T v \|^2 \geq \min(1, 1-\epsilon) (
  \|[1-\PU]v\|^2 + \| \PU v\|^2) 
  = \gamma_0 \| v \|^2,
\end{displaymath}
with $\gamma_0 = \sqrt{1-\epsilon}$.

{\it Proof of \eqref{eq:condell2:21}. } We begin by splitting the
operator, using $\Chi^2 = \Chi$, into
\begin{align*}
  \notag
  [I-\alpha Z^T L Z \Chi]
  =~&
  \Chi [I-\alpha Z^T L Z \Chi] + [I-\Chi] [I-\alpha  Z^T L Z\Chi] \\
  =~&
  \Chi[I - \alpha Z^T L Z]\Chi
  + [I-\Chi][I - \alpha Z^T L Z\Chi] \\
  =:~& S_1 + S_2.
\end{align*}
Since $\Chi$ is an orthogonal projection, we have, for any $w \in
\Us$,
\begin{equation}
  \label{eq:condell2:25}
  \big\| [I - \alpha Z^T L Z \Chi] w \big\|^2 = \| S_1 w\|^2 + \|S_2 w\|^2.
\end{equation}

{\it Estimating $S_1$. } 
Since $S_1 = \Chi [I-\alpha Z^T L Z]\Chi$ is a symmetric operator,
the following variational bound is found using the fact that $0 \leq
\< L x, x \> \leq 4.$ We obtain
\begin{align*}
  \big\< \Chi [I - \alpha Z^T L Z] \Chi w, w \big\> =~& 
  \|\Chi w \|^2 - \alpha \big\< L Z\Chi w, Z \Chi w \big\> \\
  \geq~& \min\big( 1, 1 - 4 \alpha \|Z\|^2 \big) \| \Chi w \|,
\end{align*}
where the last equality follows by distinguishing the cases $\alpha <
0$ and $\alpha \geq 0$.

In summary, if $\alpha \|Z\|^2 < 1/4$, then we have the $N,
\As$-independent bound
\begin{equation}
  \label{eq:condell2:32}
  \| S_1 w \|^2 \geq \big[\min(1, 1-4\alpha \|Z\|^2)\big]^2  \|X w\|^2
  \qquad \forall w \in \R^{N}.
\end{equation}

{\it Estimating $S_2$. } Due to the good estimate on $S_1$ we only
need fairly rough estimates on the term $\| S_2 w \|^2$. Application
of the Cauchy--Schwartz inequality and a weighted Cauchy inequality
provides the estimate
\begin{align*}
  \| S_2 w\|^2 \geq (1 - \epsilon) \| [I-\Chi] w \|^2 
  + (1-\epsilon^{-1}) \alpha^2 \| [I-\Chi] [Z^T L Z] \Chi w \|^2,
\end{align*}
for any $\epsilon \in (0, 1)$.  Using the fact that $I-\Chi$ is an
orthogonal projection, $\|L\| \leq 4$,
we can further estimate
\begin{equation}
  \label{eq:condell2:35}
  \| S_2 w\|^2 \geq (1 - \epsilon) \| [I-\Chi] w \|^2 
  + (1-\epsilon^{-1}) 16 \|Z\|^4 \alpha^2  \|\Chi w \|^2.
\end{equation}

{\it Combining the estimates. } If we define $\tilde\alpha = 4 \alpha
\|Z\|^2$, and insert \eqref{eq:condell2:32} and
\eqref{eq:condell2:35} into \eqref{eq:condell2:25}, we obtain, for
all $w \in \Us$,
\begin{align*}
  \| [I-\alpha Z^T L Z \Chi] w \|^2
  \geq~& 
  \big\{  \min(1, (1-\tilde\alpha)^2)
  + (1-\epsilon^{-1})\tilde\alpha^2 \big\} \| Xw\|^2 
  + (1-\epsilon) \|[I-X]w\|^2 \\
  \geq~&
  \min\big\{ \min(1, (1-\tilde\alpha)^2) +
  (1-\epsilon^{-1})\tilde{\alpha}^2, 1 - \epsilon \big\} \| w \|^2,
\end{align*}
for any $\epsilon \in (0, 1)$. Since $\min(1, (1-\tilde\alpha)^2) >
0$ it is clear that choosing $\epsilon$ sufficiently close to $1$
gives a positive lower bound. To optimize this constant with respect
to $\epsilon$, we have to choose $\epsilon$ to equalize the two
competitors in the outer $\min$ formula. The resulting choice is
\begin{equation}
  \label{eq:condell2:optimal_epsilon}
  \epsilon = \cases{
    \sqrt{\tilde{\alpha}^2 + \smfrac14 \tilde{\alpha}^4}
    - \smfrac12 \tilde{\alpha}^2, 
    & \tilde{\alpha} \leq 0, \\[2mm]
    \tilde\alpha - \tilde{\alpha}^2  +\sqrt{2 (\tilde{\alpha}^2 - \tilde{\alpha}^3) + \tilde{\alpha}^4}, 
    & 0 < \tilde{\alpha} < 1,
  }
\end{equation}
which concludes the proof of \eqref{eq:condell2:21}. (As a matter of
interest, $\epsilon \to 1$ as $\tilde{\alpha} \to 1$, $\epsilon = 0$
for $\tilde{\alpha} = 0$, and $\epsilon \to 1$ as $\tilde{\alpha}
\to -\infty$.)
\end{proof}

\subsection{Estimates for the eigenvalues}
\label{sec:fr:evals}
Using the similarity to a symmetric matrix that we have established in
the previous section we can now give sharp bounds on the spectrum of
$\Lqcfp$.

\begin{theorem}
\label{th:fr:evals}
Suppose that \eqref{eq:fr:b1_bound_ass} holds, then $\Lqcfp$ has a
real, ordered spectrum $(\lambda_j)_{j = 1}^{N}$. If we
denote the ordered eigenvalues of $\La $ and $\Lc$, respectively, by
$(\lambda_j^\a)_{j = 1}^{N}$ and $(\lambda_j^\c)_{j = 1}^{N}$
then
\begin{align*}
  &\lambda_j^\c \leq \lambda_j \leq \lambda_j^\a, \qquad \text{for } j = 1, \dots, N, \quad \text{if } b(t) \geq 0, \quad \text{and} \\
  &\lambda_j^\a \leq \lambda_j \leq \lambda_j^\c, \qquad \text{for } j = 1, \dots, N, \quad \text{if } b(t) \leq 0. \qedhere
\end{align*}
\end{theorem}

\begin{proof}
We know from Theorem \ref{th:fr:condbnd_abstract} that $\Lqcfp$ is
diagonalizable and that it is similar to the self-adjoint operator
$\Lsymm$ defined in \eqref{eq:fr:defn_Lsym}, which has a real
spectrum that is identical to the spectrum of $\Lqcfp$. We will
next show that, for all $u \in \Rper$,
\begin{equation}
  \label{eq:fr:evals:1}
  \begin{split}
    \< \Lc u, u, \> \leq \< \Lsymm u, u \> \leq \< \La u, u \>, & \qquad 
    \text{if } \sigma  = 1,  \quad \text{and } \\
    \< \La u, u, \> \leq \< \Lsymm u, u \> \leq \< \Lc u, u \>, & \qquad 
    \text{if } \sigma  = -1.
  \end{split}
\end{equation}
From these inequalities, the min-max characterisation of eigenvalues
\cite[Sec. XIII.1]{ReedSimonIV} immediately gives the stated result. 

To prove \eqref{eq:fr:evals:1} we will take the following starting
point:
\begin{equation}
  \label{eq:fr:evals:5}
  \< \Lsymm u, u \>
  = \< \Lc u, u \> +  \sigma \< Y^T \Chi Y u, u \> 
  = \< \Lc u, u \>  +  \sigma \< \Chi Y^T u, \Chi Y^T u \>.
\end{equation}
From here on, we treat the cases $\sigma = 1$ and $\sigma = -1$
separately.

{\it Case 1: $\sigma = 1$. } Using \eqref{eq:fr:evals:5} and the
fact that $Y^T Y = \La - \Lc$, we have
\begin{displaymath}
  \< \Lsymm u, u \> \leq \< \Lc u, u \> + \< Y^T u, Y^T u \>
  = \< \Lc u, u \> + \< [\La - \Lc] u, u \>
  = \< \La u, u \>.
\end{displaymath}
For the lower bound we use \eqref{eq:fr:evals:5} and the fact that
$\< \Chi Y^T u, \Chi Y^T u \>$ is non-negative to obtain
\begin{displaymath}
  \< \Lsymm u, u \> \geq \< \Lc u, u \>.
\end{displaymath}

{\it Case 2: $\sigma = -1$. } The idea for $\sigma = -1$ is
essentially that the roles of $\La$ and $\Lc$ are reversed. Note
that, now, $Y^T Y = \Lc - \La$. Hence, for the upper bound, we get
\begin{displaymath}
   \< \Lsymm u, u \> \leq \< \Lc u, u \>,
\end{displaymath}
whereas, for the lower bound, we get
\begin{displaymath}
  \< \Lsymm u, u \> \geq \< \Lc u, u \> - \< Y^T u, Y^T u \>
  = \< \Lc u, u \> - \< [\Lc - \La] u, u \> 
  = \< \La u, u \>. \qedhere
\end{displaymath}
\end{proof}

\subsection{The case of non-positive coefficients}
\label{sec:fr:nonpos}
For many of the common interaction potentials, such as the
Lennard--Jones potential, $\phi(r) = A r^{-12} + Br^{-6}$, or the
Morse potential, $\phi(r) = \exp(-2\alpha(r-r_0)) + 2
\exp(-\alpha(r-r_0))$, there exists a minimal strain $F_*$ such that,
for all $F \geq F_*$, we have
\begin{equation}
\label{eq:fr:nonpos_cond}
\phi_{rF}'' \leq 0 \qquad \text{for} \quad r = 2, \dots, R.
\end{equation}
In most cases, it is reasonable to assume that the macroscopic strain
$F$ remains in this region, as it would require extreme compressive
forces to violate it. Hence, for the remainder of the section, we will
assume that \eqref{eq:fr:nonpos_cond} is satisfied.
The following two lemmas reduce this case to the one studied earlier in this section.

\begin{lemma}
\label{th:p-has-real-coeffs}
Suppose that \eqref{eq:fr:nonpos_cond} holds, then $b(t) \geq 0$ in $\Ts$.
\end{lemma}
\begin{proof}
Similarly as in the proof of Lemma \ref{th:fr:defn_p1} we rewrite
$b(t)$ in the form
\begin{align}
  \notag
  b(t) =~& \sum_{r=2}^{R} \phi''_{rF} \big[(t^r-1)(t^{-r}-1) 
    - r^2 (t-1) (t^{-1}-1)\big] \\
  =~& (t-1)(t^{-1}-1) \sum_{r=2}^R \phi_{rF}'' \big[
  (t^{r-1}+t^{r-2}+\ldots+1) (t^{-r+1}+t^{-r+2}+\ldots+1) - r^2 \big].
   \label{eq:p-has-real-coeffs:intermediate_1}
\end{align}
It is easy to see that $(t-1)(t^{-1}-1)$ is non-negative on $\Ts$,
and moreover, for $t \in \Ts$,
\begin{equation}
  \label{eq:p-has-real-coeffs:intermediate_2}
  (t^{r-1}+t^{r-2}+\ldots+1) (t^{-r+1}+t^{-r+2}+\ldots+1) =
  |t^{r-1}+t^{r-2}+\ldots+1|^2 \leq r^2.
\end{equation}
Hence we obtain the stated result.
\end{proof}

\begin{lemma}
 \label{th:fr:nonpos_implicit_bound}
 Suppose that \eqref{eq:fr:nonpos_cond} holds and that
 $\phi_{RF}''<0$.  Then \eqref{eq:fr:b1_bound_ass} is satisfied with
 constants $\beta_0$ and $\beta_1$ that are independent of $N$ and
 $\As$.
\end{lemma}
\begin{proof}
 The upper bound $\beta_1$ in \eqref{eq:fr:nonpos_cond} obviously
 exists since $b_1(t)$ is a continuous function on a compact set
 $\Ts$.

 To show the existence of the lower bound it is sufficient to show
 that
 \begin{equation}
   \label{eq:fr:nonpos_implicit_lower_bound}
   b_1(t)>0 \quad \text{for all } t\in\Ts.
 \end{equation}

 Suppose that $b_1(t_1)\le 0$ at some point $t_1\in\Ts$.  Then
 $b(t_1)\le 0$ and hence at least one term in the definition of
 $b(t)$ \eqref{eq:b-definition} is non-positive. To be precise, there
 exists $r \in \{2, \dots, R\}$ such that $\phi_{rF}'' < 0$ and
 \begin{displaymath}
   (-\phi_{rF}'') \big[r^2 (t_1-1) (t_1^{-1}-1) - (t_1^r-1)(t_1^{-r}-1)\big] \le 0
 \end{displaymath}
 It follows from \eqref{eq:p-has-real-coeffs:intermediate_1} and
 \eqref{eq:p-has-real-coeffs:intermediate_2} that this may only
 happen at $t_1 = 1$. However, a striaghtforward computation shows
 that
 \begin{displaymath}
   b_1(1) = \lim_{t\to 1} \frac{r^2 (t-1) (t^{-1}-1) - (t^r-1)(t^{-r}-1)}{
     \big[(t-1) (t^{-1}-1)\big]^2}  = \frac{r^4 - r^2}{12} > 0.
 \end{displaymath}
 Hence no point $t_1\in\Ts$ such that $b_1(t_1)\le 0$ exists.
\end{proof}

\noindent We are now in a position to complete the proof of Theorem
\ref{th:fr:mainthm_simple}.

\begin{proof}[Proof of Theorem \ref{th:fr:mainthm_simple}]
 Item (i) is a special case of Proposition
 \ref{th:fr:diag_Lqcf_Lqcfp}, taking into account that, for
 non-positive coefficients, $b(t)$ defined in \eqref{eq:b-definition}
 is non-negative and $Y_1$ can therefore be chosen to be real. Item
 (ii) follows from Theorem \ref{th:fr:condbnd_abstract} and
 Proposition \ref{th:fr:nonpos_implicit_bound}. Item (iii) is
 established in Theorem \ref{th:fr:evals}.
\end{proof}

\medskip \noindent We conclude this section with a result that gives
the sharp bounds on $\beta_0$, $\beta_1$ for the case of non-positive
coefficients.

\begin{proposition}
\label{th:cor_beta_nonpos}
Let $b_1(t)$ be defined by \eqref{eq:fr:defn_b1} and suppose that
\eqref{eq:fr:nonpos_cond} holds; then \eqref{eq:fr:b1_bound_ass}
holds with constants
\begin{align*}
  \beta_0^2 =~& \sum_{r=2}^R (-\phi_{rF}'') \frac{2 r^2 + (-1)^r - 1}{8},
  \qquad \text{and} \\
  \beta_1^2 =~& \sum_{r=2}^R (-\phi_{rF}'') \frac{r^2 (r^2 - 1)}{12}.
\end{align*}
The lower bound is attained at $t=-1$ and the upper bound is attained at $t=1$.
\end{proposition}

\medskip \noindent The proof of this proposition is based on the
following technical lemma.

\begin{lemma}
\label{lem:sin-r-beta-estimate}
For $r \in \N$, $r \geq 2$, let $f_r : (0, \frac{\pi}{2}] \to \R$ be
defined as
\begin{equation}
  \label{eq:f_r-definition}
  f_r(\beta) := \frac{1}{\sin^2\beta} \left(r^2 - 
    \frac{\sin^2 r\beta}{\sin^2\beta}\right),
\end{equation}
then
\begin{align}
  \label{eq:sin-r-beta-lower-estimate}
  \inf\limits_{0<|\beta|\le \frac{\pi}{2}} f_r(\beta)
  =&
  f_r(\pi/2) = 
  r^2-\frac{1-(-1)^r}{2}, \qquad \text{and}
  \\ \label{eq:sin-r-beta-upper-estimate}
  \sup\limits_{0<|\beta|\le \frac{\pi}{2}} f_r(\beta)
  =&
  \lim\limits_{\beta\to 0} f_r(\beta)
  =
  \frac{1}{3} r^2 (r^2-1).
\end{align}
\end{lemma}
\begin{proof}
{\it Proof of \eqref{eq:sin-r-beta-lower-estimate}.}
First, notice that
\begin{displaymath}
f_r\left(\pi/2\right) = r^2 - \sin^2\frac{r\pi}{2} = r^2 - \frac{1-(-1)^r}{2},
\end{displaymath}
which proves that the left-hand side of \eqref{eq:sin-r-beta-lower-estimate} is not larger than the right-hand side.
To prove the that $f_r(\beta)\ge f_r(\pi/2)$ for all $\beta$, transform
\begin{equation}\begin{array}{rll}
\label{eq:sin-r-beta-estimate-representation}
f_r(\beta)
=&\displaystyle
\frac{1}{\sin^2\beta} \left(r^2 - \frac{\sin^2 r\beta}{\sin^2\beta}\right)
=
\frac{1}{\sin^2\beta} \left(r^2 - r^2\sin^2\beta - \frac{\sin^2 r\beta}{\sin^2\beta}\right)
+r^2
\\[1em] =&\displaystyle
\frac{1}{\sin^2\beta} \left(r^2\cos^2\beta - \frac{\sin^2 r\beta}{\sin^2\beta}\right)
+r^2
=
\frac{1}{\sin^2\beta} \left(\frac{r^2 \sin^2 2\beta}{4\sin^2\beta} - \frac{\sin^2 r\beta}{\sin^2\beta}\right)
+r^2
\\[1em] &\displaystyle\hfill
=\frac{r^2}{\sin^4\beta} \left(\frac{\sin^2 2\beta}{4} - \frac{\sin^2 r\beta}{r^2}\right)
+r^2
\end{array}
\end{equation}
and consider the three cases: $0<\beta<\frac{\pi}{2r}$, $\frac{\pi}{2r}\le\beta\le\frac{\pi}{2}-\frac{\pi}{2r}$, and $\frac{\pi}{2}-\frac{\pi}{2r}<\beta\le 1$.

{\it Case 1.} ($0<\beta\le\frac{\pi}{2r}$)
Further transform the function $f_r(\beta)$ in \eqref{eq:sin-r-beta-estimate-representation}:
\begin{align*}
f_r(\beta)
=
\frac{r^2\beta^2}{\sin^4\beta} \left(\sinc^2 2\beta - \sinc^2 r\beta\right)
+r^2.
\end{align*}
The expression in the brackets is positive since $\sinc x = \frac{\sin x}{x}$ is a decreasing function for $0<x\le\pi/2$.
This proves $f_r(\beta)\ge r^2\ge f_r(\pi/2)$.

{\it Case 2.} ($\frac{\pi}{2r}\le\beta\le\frac{\pi}{r}-\frac{\pi}{2r}$)
In this case $\sin 2\beta\ge \frac{2}{r}$, hence
\begin{align*}
f_r(\beta)
=~&
\frac{1}{\sin^2\beta} \left(\frac{r^2 \sin^2 2\beta}{4\sin^2\beta} - \frac{\sin^2 r\beta}{\sin^2\beta}\right)
+r^2
\\ \ge~&
\frac{1}{\sin^2\beta} \left(\frac{1}{\sin^2\beta} - \frac{\sin^2 r\beta}{\sin^2\beta}\right)
+r^2
\ge r^2 \ge f_r(\pi/2).
\end{align*}

{\it Case 3.} ($\frac{\pi}{2}-\frac{\pi}{2r}<\beta\le 1$)
Denote $x=\frac{\pi}{2}-\beta$ ($0\le x<\frac{\pi}{2r}$) and consider the two different subcases: $r$ is even/odd.

{\it Case 3.1.} ($r$ is even)
In this case $\sin\beta = \cos x$,  $\sin 2\beta = \sin 2x$, and $\sin^2 r\beta = \sin^2(r x)$.
Hence $f_r(\theta)$ as expressed in \eqref{eq:sin-r-beta-estimate-representation} takes the form
\begin{align*}
f_r(\beta)
=~&
\frac{r^2}{\cos^4 x} \left(\frac{\sin^2 2x}{4} - \frac{\sin^2 rx}{r^2}\right)
+r^2
\\ =~&
\frac{r^2 x^2}{\cos^4 x} \left(\sinc^2 2x - \sinc^2 rx\right)
+r^2
\ge r^2 = f_r(\pi/2).
\end{align*}

{\it Case 3.2.} ($r$ is odd)
In this case $\sin\beta = \cos x$,  $\sin^2 2\beta = \sin^2 2x$, and $\sin^2 r\beta = \cos^2(r x)$.
Hence \eqref{eq:sin-r-beta-estimate-representation} transforms into 
\begin{align*}
f_r(\beta)
=~&
\frac{r^2}{\cos^4 x} \left(\frac{\sin^2 2x}{4} - \frac{\cos^2 rx}{r^2}\right)
+r^2
\ge
r^2-\left(\frac{\cos^2 x}{\cos rx}\right)^{-2}
=
r^2-\left(\frac{1+\cos 2x}{2 \cos rx}\right)^{-2}
\\ \ge~&
r^2-\left(\frac{1}{2} + \frac{1}{2}\right)^{-2}
= r^2-1 = f_r(\pi/2).
\end{align*}
Here we used the fact that $1\ge\cos 2x\ge\cos rx$ ($0\le x<\frac{\pi}{2r}$).

{\it Proof of \eqref{eq:sin-r-beta-upper-estimate}.}
First compute
\begin{align*}
\lim\limits_{\beta\to 0} f_r(\beta)
=~&
\lim\limits_{\beta\to 0}\frac{r^2 \sin^2 \beta - \sin^2 r\beta}{\sin^4\beta}
= \lim\limits_{\beta\to 0}\frac{r^2 \left(\beta^2-\beta^4/3\right) - \left(r^2 \beta^2-r^4 \beta^4/3\right)+ O(\beta^6)}{\beta^4}
\\ =~&
-\frac{r^2}{3}+\frac{r^4}{3}
= \frac{1}{3} r^2 (r^2-1).
\end{align*}
To prove the inequality $f_r(\beta) \le \frac{1}{3} r^2 (r^2-1)$ consider the two cases: $\frac{\pi}{r}\le\beta\le\frac{\pi}{2}$ and $0<\beta<\frac{\pi}{r}$.

{\it Case 1.} ($\frac{\pi}{r}\le\beta\le\frac{\pi}{2}$)
In this case \eqref{eq:sin-r-beta-upper-estimate} follows directly from the following computation:
\begin{align*}
\frac{1}{3} r^2 (r^2-1) \big[f_r(\beta)\big]^{-1}
\ge~&
\frac{1}{3} r^2 (r^2-1) \left[\frac{r^2}{\sin^2{\beta}}\right]^{-1}
=
\frac{1}{3} (r^2-1)\,\sin^2{\beta}
\\ \ge~&
\frac{1}{3} (r^2-1)\,\sin^2 \frac{\pi}{r}
\ge
\frac{1}{3} (r^2-1)\, \left(\frac{2}{r}\right)^2
=
\frac{4 r^2-4}{3 r^2}
\ge
\frac{4}{3} - \frac{1}{3} = 1.
\end{align*}

{\it Case 2.} ($0<\beta<\frac{\pi}{r}$)
We need to prove
\begin{equation}
\label{eq:sin-r-beta-upper-estimate:bound}
r^{-2} \left(f_r(\beta) - \frac{1}{3} r^2 (r^2-1)\right)
~=~
\frac{f_r(\beta)}{r^2} - \frac{1}{3} (r^2-1)
~\le~ 0
\qquad (\forall\beta\in(0,\pi/r))
\end{equation}
for integer $r$, but instead we prove that it is valid for all real values of $r\in [2,\infty)$.

First, notice that the following calculation
\begin{displaymath}
\left.\frac{f_r(\beta)}{r^2} - \frac{1}{3} (r^2-1)\right|_{r=2}
= \frac{1}{4}\,\frac{1}{\sin^2\beta} \left(4 - \frac{\sin^2 2\beta}{\sin^2\beta}\right)-1
=
\frac{1}{4 \sin^2\beta} \left(4 - 4 \cos^2\beta\right) - 1
= 0
\end{displaymath}
proves \eqref{eq:sin-r-beta-upper-estimate:bound} for $r=2$.
As is shown below,
\begin{equation}
\label{eq:sin-r-beta-upper-estimate:derivative-estimate}
\frac{\partial}{\partial r}\left(\frac{f_r(\beta)}{r^2} - \frac{1}{3} (r^2-1)\right) \le 0
\qquad (\forall\beta\in(0,\pi/r)),
\end{equation}
which concludes the proof of \eqref{eq:sin-r-beta-upper-estimate:bound}.

{\it Proof of \eqref{eq:sin-r-beta-upper-estimate:derivative-estimate}.}
We use the standard inequalities
\begin{align*}
\sinc'(x)\ge -x/3,
\qquad& \forall x\ge 0
\\
\cos x\le \sinc^3(x)
\qquad& \forall x\in [0,\pi]
\end{align*}
and obtain
\begin{align*}
\frac{\partial}{\partial r}\left(\frac{f_r(\beta)}{r^2}\right)
=~&
\frac{\partial}{\partial r}\left(\frac{1}{\sin^2\beta} - \frac{\sin^2{r \beta}}{r^2 \sin^4{\beta}}\right)
=
- \frac{\partial}{\partial r}\left(\frac{\beta^2 \sinc^2{r \beta}}{\sin^4{\beta}}\right)
\\ =~&
- \frac{2 \beta^3 \sinc{r \beta} \sinc'{r \beta}}{\sin^4\beta}
\le
\frac{2 \beta^3 \sinc{r \beta}\,(r \beta)}{3 \sin^4\beta},
\end{align*}
and hence
\begin{align*}
\frac{\partial}{\partial r}\left(\frac{f_r(\beta)}{r^2} - \frac{1}{3} (r^2-1)\right)
\le~&
\frac{2 \beta^3 \sinc{r \beta}\,(r \beta)}{3 \sin^4\beta} - \frac{2 r}{3}
=
\frac{2 r}{3} \left(\frac{\beta^4 \sinc{r \beta}}{\sin^4\beta}-1\right)
\\ \le~&
\frac{2 r}{3} \left(\frac{\beta^4 \sinc{2 \beta}}{\sin^4\beta}-1\right)
=
\frac{2 r}{3} \left(\frac{\beta^3 \cos{\beta}}{\sin^3\beta}-1\right)
\le 0.
\qedhere
\end{align*}
\end{proof}

\medskip 

\begin{proof}[Proof of Proposition \ref{th:cor_beta_nonpos}]

To reduce the problem to the statement of Lemma
\ref{lem:sin-r-beta-estimate} we make the substitution $\beta = e^{2 i \beta}$, $-\pi/2 < \theta < \pi/2$, which gives
\begin{align*}
  b(t) =~&
  \sum\limits_{r=2}^R (-\phi_{rF}'')\,
  \frac{r^2 (t-1) (t^{-1}-1) - (t^r-1) (t^{-r}-1)}{\big[(t-1) (t^{-1}-1)\big]^2}
  \\ =~&
  \sum\limits_{r=2}^R (-\phi_{rF}'')\,
  \frac{r^2 (e^{2i\beta}-1) (e^{-2i\beta}-1) - (e^{2ir\beta}-1) (e^{-2ir\beta}-1)}{\big[(e^{2i\beta}-1) (e^{-2i\beta}-1)\big]^2}
  \\ =~&
  \sum\limits_{r=2}^R (-\phi_{rF}'')\,
  \frac{r^2\, 4 \sin^2{\beta} - 4 \sin^2{r\beta}}{16 \sin^4{\beta}}
  =
  \sum\limits_{r=2}^R (-\phi_{rF}'')\,
  \frac{f_r(\beta)}{4}
  .
\end{align*}
Application of Lemma \ref{lem:sin-r-beta-estimate} gives Proposition
\ref{th:cor_beta_nonpos}.
\end{proof}

\section{Analysis of Preconditioned $\Lqcfp$ Operators}
\label{sec:prec}
In this final section we present two further interesting applications
of our foregoing analysis. First, we discuss the GMRES solution of a
linearized QCF system. We rigorously establish some conjectures used
in \cite{Dobson:qcf.iter} and briefly discuss their
consequences. Second, we prove a new stability result for the
linearized QCF operator in a discrete Sobolev norm, which we hope will
become a useful tool for future analyses of the QCF method.

We assume throughout this section that \eqref{eq:fr:p1_bounds} holds
and that $Y_1$ is defined by \eqref{eq:fr:new_Y1}. Moreover, we recall
the definition of $\Lsymm$ from \eqref{eq:fr:defn_Lsym}. Since the
results are fairly straightforward corollaries from our analysis in
Section \ref{sec:fr} we will derive them in a less formal manner.

\subsection{GMRES-Solution of the QCF system}
We consider the linearized QCF system
\begin{equation}
 \label{eq:prec:qcf_sys}
 \Lqcfp u = f,
\end{equation}
where $f \in \Us$, which is to be solved for $u \in \Us$. If this
system is solved using the GMRES algorithm (see \cite{saad03} for a
general introduction and \cite{Dobson:qcf.iter} for a detailed
discussion of using GMRES for solving the QCF system), then standard
estimates on GMRES convergence~\cite{saad03}, along with the analysis
of the previous sections, show that the residual of the $m$-th
iterate, $r^{(m)} = f - \Lqcfp u^{(m)}$, satisfies the bound
\begin{displaymath}
 \| r^{(m)} \| \leq 2\cond(\Vqcf)
 \bigg(\frac{1-\sqrt{\gamma}}{1+\sqrt{\gamma}}\bigg)^m \, \|r^{(0)}\|,
\end{displaymath}
where $\gamma = \lambda_2 / \lambda_N = O(1/N^2)$ (see also
\cite[Prop. 22]{Dobson:qcf.iter}). The fraction $\lambda_2/\lambda_N$
is used instead of $\lambda_1/\lambda_N$ since we are solving the
system in $\Us$. This convergence is rather slow and hence two
variants of preconditioned GMRES (P-GMRES) algorithms were suggested
in \cite{Dobson:qcf.iter}. The preconditioner used was the negative
Laplacian $L$. We will use the preconditioner $\Lmod$ instead of $L$,
however, this is purely for the sake of a consistent notation since
$\Lmod^{-1} \Lqcfp = L^{-1} \Lqcfp$ (note that $\range \Lqcfp = \Us$
and that $L^{-1}$ is well-defined on $\Us$).

The first variant of P-GMRES that was considered in
\cite{Dobson:qcf.iter} was the standard left-precon- ditioned GMRES
algorithm where GMRES is applied to the preconditioned system
\begin{equation}
 \label{eq:prec:leftprec_sys}
 \Lmod^{-1} \Lqcfp u = \Lmod^{-1} f.
\end{equation}
To obtain convergence rates, we require bounds on the eigenvalues and
eigenbasis of $\Lmod^{-1} \Lqcfp$ (see
\cite[Sec. 6.2]{Dobson:qcf.iter}).

The second variant considers again the left-preconditioned system
\eqref{eq:prec:leftprec_sys} but this time the residual is minimized
in the norm induced by the operator $\Lmod$. The convergence rates of
the resulting method are governed by the spectrum and eigenbasis of
the operator $\Lmod^{-1/2} \Lqcfp \Lmod^{-1/2}$ (see
\cite[Sec. 6.3]{Dobson:qcf.iter}).

\subsubsection{Diagonalization}
We consider $\Lmod^{-1/2} \Lqcfp \Lmod^{-1/2}$ first. Using
\eqref{eq:fr:defn_Lsym}, and the fact that $\Lmod^{-1/2}$ commutes
with $\Lmod Y_1$, we obtain
\begin{displaymath}
 [\Lmod Y_1] [\Lmod^{-1/2} \Lqcfp \Lmod^{-1/2}] [\Lmod Y_1]^{-1}
 = \Lmod^{-1/2} \Lsymm \Lmod^{-1/2} = \Wsymm \Lamtil (\Wsymm)^T,
\end{displaymath}
where $\Lamtil$ is the real diagonal matrix of eigenvalues and
$\Wsymm$ an orthonormal matrix of eigenvectors of $\Lmod^{-1/2} \Lsymm
\Lmod^{-1/2}$. Hence, we conclude that $\Lmod^{-1/2} \Lqcfp
\Lmod^{-1/2}$ is diagonalizable with real spectrum $\Lamtil$:
\begin{equation}
 \label{eq:prec:decomp_leftright}
 [\Lmod^{-1/2} \Lqcfp \Lmod^{-1/2}] [Y_1^{-1} \Lmod^{-1} \Wsymm] = 
[Y_1^{-1} \Lmod^{-1} \Wsymm] \Lamtil.
\end{equation}
Multiplying the equation by $\Lmod^{-1/2}$, we obtain
\begin{equation}
 \label{eq:prec:decomp_left}
 [\Lmod^{-1} \Lqcfp] [Y_1 \Lmod^{-3/2} \Wsymm]
 = [Y_1^{-1} \Lmod^{-3/2} \Wsymm] \Lamtil,
\end{equation}
thus showing that also $\Lmod^{-1} \Lqcfp$ is diagonalizable with the same
real spectrum $\Lamtil$. We note that this rigorously establishes a
variant of \cite[Conjecture 10]{Dobson:qcf.iter}.

\subsubsection{Condition number bounds}
Using the fact that $\Wsymm$ is orthogonal, and Lemma
\ref{th:fr:bounds_Y1} to bound $\cond(Y_1) \leq \beta_1/\beta_0$, we
can obtain the following upper bounds on the condition number of the
matrices of eigenvectors:
\begin{align}
 \label{eq:prec:cond_leftright}
 \cond(Y_1^{-1} \Lmod^{-1} \Wsymm) \lesssim~& N^2 \beta_1 / \beta_0, \quad \text{and} \\
 \label{eq:prec:cond_left}
 \cond(Y_1^{-1} \Lmod^{-3/2} \Wsymm) \lesssim~& N^{3/2} \beta_1 / \beta_0. 
\end{align}
This rigorously establishes \cite[Conjectures 12 and
13]{Dobson:qcf.iter}.

Since these bounds are not uniform in $N$ the question arises whether
we can define a better scaling for the eigenvectors to improve
them. Note, however, that \cite[Thm. 4.3]{Dobson:qcf.stab} implies
that $\cond(\Lqcfp) \gtrsim N^{1/2}$, and hence no choice of
eigenbasis can achieve an upper bound in
\eqref{eq:prec:cond_leftright} that is uniform in $N$. Moreover,
numerical experiments in \cite[Sec. 3, Figs. 2 and 3]{Dobson:qcf.iter}
indicate that our bounds may be optimal.

\subsubsection{Eigenvalue bounds}
To establish convergence rates for the P-GMRES solution of the QCF
system, we still need to obtain bounds on the eigenvalues contained in
$\Lamtil$. Let $(\tilde\lambda_n)_{n = 1}^N$ denote the ordered
eigenvalues of $\Lamtil$, and let $(\tilde\lambda_n^\a)_{n = 1}^N$ and
$(\tilde\lambda_n^\c)_{n = 1}^N$ denote, respectively, the ordered
eigenvalues of $\La$ and $\Lc$. Since $\Lc = W_F'' L$, we know that
\begin{displaymath}
 \tilde\lambda_1^\c = 0, \quad \text{and} \quad
 \tilde{\lambda}_n^\c = W_F'' \quad \text{for } n = 2, \dots, N.
\end{displaymath}
In view of Remark \ref{rem:sign_b}, replacing $u$ by $\Lmod^{-1/2}
u$ in the formulas, we obtain that
\begin{align*}
 \text{either}& \qquad \big\< [\Lmod^{-1/2} \La \Lmod^{-1/2}] u, u \big\>
 \geq \big\< [\Lmod^{-1/2} \Lc \Lmod^{-1/2}] u, u \big\> \qquad \forall u \in \Us,  \\
 \text{or}& \qquad \big\< [\Lmod^{-1/2} \La \Lmod^{-1/2}] u, u \big\>
 \leq \big\< [\Lmod^{-1/2} \Lc \Lmod^{-1/2}] u, u \big\> \qquad \forall u \in \Us.
\end{align*}
Hence, we can repeat the proof of Theorem \ref{th:fr:evals} verbatim to show that
\begin{displaymath}
 \min(\tilde\lambda_n^\a, \tilde\lambda_n^\c) \leq
 \tilde\lambda_n 
 \leq \max(\tilde\lambda_n^\a, \tilde\lambda_n^\c)
 \qquad \text{for } n = 1, \dots, N.
\end{displaymath}

At this point we need to make an assumption on the stability of the
atomistic system. We assume that the macroscopic strain $F$ is chosen
so that
\begin{equation}
 \label{eq:prec:stab_bnd_La}
  c_0 \|u'\|^2 \leq \< \La u, u \> \leq c_1 \|u'\|^2
  \qquad \forall u \in \Us.
\end{equation}
The upper bound can be obtained in a straightforward computation that
gives a constant $c_1$ depending only on the coefficients
$\phi_{rF}''$, $r = 1, \dots, R$. The lower bound means that the
homogeneous deformation $Fx$ lies in the region of stability of the
atomistic energy (see \cite{Dobson:qce.stab} for a detailed discussion
of this point, in particular, that $c_0$ is indeed independent of
$N$). For example, in the case of non-positive coefficients Remark
\ref{rem:sign_b} shows that this bound holds with $c_0 = W_F''$, and
that in the case $R = 2$ one can choose $c_0 = \min(W_F'', \phi_F'')$.

Upon noting that the stability assumption \eqref{eq:prec:stab_bnd_La}
is equivalent to the statement that $\tilde\lambda_2^\a \geq c_0$ and
$\tilde\lambda_N^\a \leq c_1$, we can now deduce that
\begin{displaymath}
 \min(c_0, W_F'') \leq \tilde\lambda_n \leq \max(c_1, W_F''),
\end{displaymath}
which are bounds that are independent of $N$ and $\As$.

\subsubsection{Convergence rates for P-GMRES}
From the foregoing discussion we obtain the following convergence
rates for the P-GMRES solution of \eqref{eq:prec:qcf_sys} (see
\cite[Sec. 6.2 and Sec. 6.3]{Dobson:qcf.iter} for details of these
derivations):

For the standard left-preconditioned GMRES algorithm we obtain bounds
on the preconditioned residual,
\begin{equation}
 \label{eq:prec:convrate_left}
 \| \Lmod^{-1} r^{(m)}\| \leq C N^3 q^m \|\Lmod^{-1} r^{(0)} \|,
\end{equation}
where $C > 0$ and $q \in (0, 1)$ are independent of $N$ and $\As$.

For the left-preconditioned P-GMRES algorithm, which minimizes the
preconditioned residual in the norm induced by $\Lmod$, we obtain
\begin{equation}
 \label{eq:prec:convrate_leftright}
 \| \Lmod^{-1/2} r^{(m)}\| \leq C N^2 q^m \|\Lmod^{-1/2} r^{(0)} \|,
\end{equation}
where $C \geq 0$ and $q \in (0, 1)$ are independent of $N$ and $\As$.

We also note that a finer analysis (see \cite[Sec. 6.2 and
6.3]{Dobson:qcf.iter}) shows that both variants of P-GMRES reduce the
residual to zero in at most $O(\#\As)$ iterations.

\subsection{Stability of $\Lqcfp$ in $\Us^{2,2}$}
We define the discrete Sobolev-type norm
\begin{displaymath}
 \| u \|_2 = \| L u \| \qquad \text{for } u \in \Rper,
\end{displaymath}
which is a norm on the space $\Us$ of mean-zero functions and denote
the space $\Us$ equipped with $\|\cdot\|_2$ by $\Us^{2,2}$. Moreover,
we denote the space $\Us$ equipped with the norm $\|\cdot\| =:
\|\cdot\|_0$ by $\Us^{0,2}$.  We are interested in the question
whether $\Lqcfp : \Us^{2,2} \to \Us^{0,2}$ is stable, uniformly in $N$
and $\As$.

To begin with, we note that
\begin{displaymath}
 \| (\Lqcfp)^{-1} \|_{L(\Us^{0,2}, \Us^{2,2})}^{-1} 
 = \inf_{u \in \Us \setminus \{0\}} \frac{\| \Lqcfp u \|_0}{\|Lu\|_0}
 = \inf_{f \in \Us \setminus \{0\}} \frac{ \| \Lqcfp \Lmod^{-1} f \|_0}{\|f\|_0};
\end{displaymath}
thus, the question reduces to the analysis of the
operator $\Lqcfp \Lmod^{-1}$. Using the representation
\begin{displaymath}
 \Lqcfp = \Lc + \sigma \PU \Chi (\La - \Lc) = W_F'' L + \sigma \PU \Chi Y_1^T Y_1 L^2,
\end{displaymath}
where $Y_1$ is defined in \eqref{eq:fr:new_Y1}, we obtain
\begin{displaymath}
 \Lqcfp \Lmod^{-1} = W_F'' \PU + \sigma \PU \Chi Y_1^T Y_1 L.
\end{displaymath}
We now argue similar as in the proof of Theorem
\ref{th:fr:bounds_Y1}. On the space $\Us$ the operator $W_F'' \PU +
\PU \Chi Y_1^T Y_1 L$ can be replaced by
\begin{displaymath}
 W_F'' I + \sigma \PU \Chi Y_1^T Y_1 L.
\end{displaymath}
For this modified operator Lemma \ref{th:fr:condell2_lowerbd} shows
that it is invertible and provides uniform bounds on the
inverse. Restricting the argument back to $\Us$ we obtain the
following theorem.

\begin{theorem}
 \label{th:prec:U22stab}
 Suppose that $W_F'' > 0$, that \eqref{eq:fr:b1_bound_ass} holds, and
 that $\sigma\beta_1^2/W_F'' > -1/4$; then $\Lqcfp$ is invertible and
 $\| (\Lqcfp)^{-1} \|_{L(\Us^{0,2}, \Us^{2,2})}$ is bounded above by
 a constant that depends only on $W_F'', \beta_0, \beta_1$ (that is,
 on the coefficients $\phi_{rF}'', r = 1, \dots, R$) but is
 independent of $N$ and of the choice of $\As$.
\end{theorem}

\section{Conclusion}
We have established a comprehensive $\ell^2$-theory of a linearized
force-based quasicontinuum (QCF) operator making several conjectures
from previous work \cite{Dobson:qcf.iter,Dobson:qcf.stab} regarding
its spectrum and eigenbasis rigorous (at least up to a modification of
the boundary conditions). We have given elementary derivations in the
case of next-nearest neighbour interactions but have also provided
proofs for arbitrary finite range interactions. Finally, as an
immediate corollary of our analysis we have also obtained a new
stability estimate in the space $\Us^{2,2}$.

Our results heavily use the fact that the nonlinear QCF operator is
linearized at a homogeneous deformation and a question of immediate
relevance is whether our results can be generalized, at least
partially, to linearizations around non-uniform states. Even in small
neighbourhoods of homogeneous deformations it is unclear whether this
can be done. 

Secondly, a generalization to 2D or 3D would have immense consequences
as no approach to the analysis of QCF in 2D or 3D exists at this point.

\end{document}